\documentclass[10pt, letterpaper, twoside]{article} 

\usepackage{amsmath, amsthm, amssymb,bbold,mathrsfs} 
\usepackage{graphicx}
\usepackage[titletoc,title]{appendix}
\usepackage{enumitem}
\setlist[itemize]{itemsep=0pt,parsep=2pt,topsep=2pt}
\setlist[enumerate]{itemsep=0pt,parsep=2pt,topsep=2pt}
\usepackage{thmtools}
\usepackage{thm-restate}
\usepackage[numbers]{natbib}

\usepackage[hidelinks]{hyperref}

\usepackage{cleveref}

\newtheorem{theorem}{Theorem}[section]
\newtheorem{lem}[theorem]{Lemma}%[section]
\newtheorem{prop}[theorem]{Proposition}%[section]
\newtheorem{assumption}[theorem]{Assumption}%[section]

\theoremstyle{remark}
\newtheorem{rem}{Remark}[section]

\theoremstyle{definition}
\newtheorem{example}{Example}[section]
\newtheorem{definition}{Definition}[section]

\numberwithin{equation}{section}
 \def\argmin{\mathop{\arg\min}}
 \def\A{\mathbb{A}}
 \def\X{\mathbb{X}}
 \def\E{\mathbb{E}}
 
 \def\R{\mathbb{R}}
 \def\Pr{\mathbb{P}}

 \def\cE{\mathcal{E}}
 \def\F{\mathcal{F}}
 \def\P{\mathcal{P}}
 \def\B{\mathcal{B}}
 \def\M{\mathcal{M}}
 \def\U{\mathcal{U}}
 \def\S{\mathcal{S}}
 \def\Z{\mathcal{Z}}

 \def\mrho{\kappa} %{\rho}
 \def\ugam{\underline{\gamma}}
 \def\nmid{\,|\,}
\def\leb{\lambda_{\text{\tiny Leb}}}

 \def\0{\mathbf{0}}
 \def\1{\mathbf{1}}
 \def\ind{\mathbb{1}}
 
 \def\mdp{$\text{AC}^+$}
 \def\mdn{$\text{AC}^-$}
 \def\ptmdp{$\widetilde{\text{AC}}^+$}
 \def\ptmdn{$\widetilde{\text{AC}}^-$}
 \def\tJ{{\tilde J}}
 \def\tg{{\tilde g}}

\def\old#1{{}}
 
\begin{document} 
\markboth{Average-Cost Borel-Space MDPs}{Average-Cost Borel-Space MDPs}

\title{Average-Cost Optimality Results for Borel-Space Markov Decision Processes with Universally Measurable Policies\thanks{This research was supported by grants from DeepMind, Alberta Machine Intelligence Institute (AMII), and Alberta Innovates---Technology Futures (AITF).}}

\author{Huizhen Yu\thanks{RLAI Lab, Department of Computing Science, University of Alberta, Canada (\texttt{janey.hzyu@gmail.com})}}
\date{}

\maketitle

\begin{abstract}
We consider discrete-time Markov Decision Processes with Borel state and action spaces and universally measurable policies.
For several long-run average cost criteria, we establish the following optimality results: the optimal average cost functions are lower semianalytic, there exist universally measurable semi-Markov or history-dependent $\epsilon$-optimal policies, and similar results hold for the minimum average costs achievable by Markov or stationary policies. 
We then analyze the structure of the optimal average cost functions, proving sufficient conditions for them to be constant almost everywhere with respect to certain $\sigma$-finite measures. 
The most important condition here is that each subset of states with positive measure be reachable with probability one under some policy. 
We obtain our results by exploiting an inequality for the optimal average cost functions and its connection with submartingales, and, in a special case that involves stationary policies, also by using the theory of recurrent Markov chains.
\end{abstract}

submartingales; reachability; recurrent Markov chains

\clearpage
\tableofcontents

\clearpage
\section{Introduction}

We study discrete-time Markov Decision Processes (MDPs) on Borel state and action spaces under long-run expected average cost criteria. 
In particular, we consider Borel-space MDPs as formulated in \cite{ShrB78,ShrB79} and \cite[Part II]{bs}, where the control constraints have analytic graphs, the one-stage cost functions are lower semianalytic, and the policies are universally measurable. 
This mathematical framework is a result of a series of research efforts, starting with the early work \cite{Blk-borel,BFO74,Str-negative}, to resolve measurability issues in Borel-space MDPs, and is applicable to modeling a broad range of control systems with complex dynamics and objective functions. The purpose of this work is to investigate the average-cost optimality properties of MDPs in this general framework.

In contrast to finite-space MDPs, the theory of infinite-space MDPs is incomplete, with the average cost problems being especially hard to analyze.
While there are extensive studies on average-cost Borel-space MDPs that satisfy various continuity/compactness conditions (cf.\ \cite{ArB19,CoD12,FKL20,FKZ12,HL96,HL99,HLV98,JaN06,Mey97,Sch92,VAm18} and the references therein), for more general MDP models, there are not many published results. Of these, the majority are applicable only in the case where the optimal average cost function is constant on the state space (see e.g., \cite{GuS75,Jas09,Kur86,Yu20}). To the best of our knowledge, the only available method for non-constant optimal average cost functions is the ``canonical triplet'' approach (\cite[Chap.\ 7.9]{DyY79}, \cite{Piu89}; see also \cite[Chap.~5.2]{HL96}). 
A canonical triplet consists of two functions on the state space and a stationary policy that together satisfy certain constraints. 
Such a triplet, if it exists, provides the optimal average cost function, a solution to a pair of average-cost optimality equations (ACOEs), as well as a stationary ($\epsilon$-)optimal policy for an MDP. 
On the other hand, even for a countable-space MDP, the ACOE need not admit a solution \cite{CCa91} and a stationary ($\epsilon$-)optimal policy need not exist \cite{Fei80,FiR68}. 
In summary, although for particular classes of MDPs, the average cost problems have been studied in depth,
general average-cost MDPs are challenging to analyze and their optimality properties are not fully understood.

In this work, we examine the structural properties of the optimal average cost functions and the ($\epsilon$-)optimal policies, starting from basic properties such as measurability. 
We are not concerned with the validity of optimality conditions of the dynamic-programming (DP) type, such as ACOEs, which, as noted above, may not hold in general. Instead, by approaching the average cost problems in a ``bottom-up'' way, 
we seek to find structures that can help shed light on the nature of these problems and enhance the existing theory.

The main contributions of this paper are twofold.
First, we establish basic optimality results for several average cost criteria, including the standard average cost criterion, a criterion that is similar to but stronger than the standard one, and some criteria that are based on average costs along sample paths (see  (\ref{eq-crt1})-(\ref{eq-crt3}) and (\ref{eq-pt-ac1})-(\ref{eq-pt-ac2a})).
In particular, under general conditions that ensure the average cost functions of all policies are well defined, we prove that the optimal average cost functions are lower semianalytic and that there exist universally measurable, randomized semi-Markov or history-dependent $\epsilon$-optimal policies. In addition, we also analyze the minimum average costs achievable by Markov or stationary policies and prove similar optimality results. (See Theorems~\ref{thm-strat-m}-\ref{thm-ac-basic2} in Section~\ref{sec-basic-opt}.) Our analyses of these basic optimality properties can also be applied to risk-sensitive MDPs with universally measurable policies. (We compare our results with the related prior work in Remark~\ref{rmk-basicthm-compare}, and we provide a detailed discussion of the proof techniques in Remark~\ref{rmk-basic-thm-prf}.)

Second, we prove sufficient conditions for the optimal average cost functions to be constant almost everywhere on certain subsets of states, with respect to certain $\sigma$-finite measures
(see Theorems~\ref{thm-ac-const}-\ref{thm-ac-const2} and Prop.~\ref{prp-ac-const-mc} in Section~\ref{sec-ae-const}).
The key condition we introduce here is, roughly speaking, that each subset of states with positive measure be reachable with probability one under some policy (see condition (i) of Theorem~\ref{thm-ac-const}). 
When specialized to the case of finite state and action spaces, this condition is closely related to the definition of a connected class \cite{Pla77} which is the basis for defining weakly communicating and multichain MDPs on finite spaces. However, for infinite spaces, this condition differs in essential aspects from a connected class (see Remark~\ref{rmk-conn-cls} for details). 
To prove our results, we use submartingale-based arguments. We exploit the fact that under suitable boundedness conditions on the one-stage costs, an inequality satisfied by the optimal average cost functions implies that the values of an optimal average cost function at the states visited under any policy form a submartingale sequence. 

Furthermore, we study an important special case of the key condition mentioned above, where the policies involved are stationary. We show how the theory of $\psi$-irreducible recurrent Markov chains can be used in this case to relate this condition to a recurrence condition and to help characterize the structure of the optimal average cost functions (see Lemma~\ref{lem-ac-cond-mc} and Remarks~\ref{rmk-mc-related}-\ref{rmk-mc-related2} in Section~\ref{sec-ae-const}). Compared with many prior studies on average-cost MDPs (e.g., \cite{ArB19,HL99,HMC91,HLV98,Mey97}), the recurrence condition employed in our result is much weaker: for example, positive recurrence is not required (see Section~\ref{sec-3.2.3} for a further discussion).

The rest of the paper is organized as follows. In Section~\ref{sec-2} we describe the mathematical framework for Borel-space MDPs and several types of average cost criteria.
We then present our results and illustrative examples in Section~\ref{sec-results} and give the proofs in Sections~\ref{sec-proofs-basic} and~\ref{sec-5}. A brief review of standard terminology for Markov chains is included in Appendix~\ref{appsec-mc}.

\section{Preliminaries} \label{sec-2}

We begin with the definitions of certain sets/functions that underly the Borel-space MDP framework. 

\subsection{Some Notation and Definitions} \label{sec-2.1}
Let $X$ be a separable metrizable space homeomorphic to a Borel subset of some Polish space; such a space is called a \emph{Borel space} (or \emph{standard Borel space}) \cite[Def.\ 7.7]{bs}. 
Let $\B(X)$ denote the Borel $\sigma$-algebra on $\X$. 
A probability measure on $\B(\X)$ will be called a \emph{Borel probability measure}. 
The set of all such measures is denoted by $\P(\X)$; endowed with the topology of weak convergence, $\P(X)$ is also a Borel space~\cite[Chap.\ 7.4]{bs}. For each $p \in \P(X)$, the \emph{completion of $p$} is the unique extension of $p$ on the $\sigma$-algebra $\B_p(X)$ generated by $\B(X)$ and all the subsets of $X$ with $p$-outer measure $0$, and it coincides with the outer measure of $p$ on $\B_p(X)$ (cf.\ \cite[Chap.\ 3.3]{Dud02}). Notation-wise, we do not distinguish between $p$ and its completion, except when their difference matters in an analysis.
The \emph{universal $\sigma$-algebra} on $X$ is given by $\U(X) : = \cap_{p \in \P(X)} \B_p(X)$ and thus contains $\B(\X)$.
The sets in $\U(X)$ and the $\U(X)$-measurable mappings on $X$ are called \emph{universally measurable}---they are
measurable with respect to (w.r.t.)~the completion of any $p \in \P(X)$. 

A subset of $X$ is called \emph{analytic}, if it is either the empty set or the image of a Borel subset of some Polish space under a Borel measurable mapping (cf.\ \cite[Prop.\ 7.41]{bs}, \cite[Chap.~13.2]{Dud02}). A \emph{lower semianalytic} function is a function $f: D \to [-\infty, \infty]$ such that the domain $D$ is an analytic set and for every $r \in \R$, the level set $\{ x \in D \!\mid f(x) \leq r\}$ is analytic \cite[Def.\ 7.21]{bs}. 
All Borel subsets of $X$ are analytic, and all Borel measurable extended real-valued functions on $X$ are lower semianalytic, whereas all analytic subsets of $X$ and lower semianalytic functions on $X$ are universally measurable.
Analytic sets and lower semianalytic functions play foundational roles in the mathematical framework for Borel-space MDPs (cf.\ the article \cite{BFO74} and the monograph \cite[Chap.\ 7]{bs}). A brief review of their properties will be given later in Section~\ref{sec-proof-review}.

If $X$ and $Y$ are Borel spaces, a function $q(\cdot \,|\, \cdot): \B(Y) \times X \to [0,1]$ is called a \emph{universally measurable stochastic kernel} (resp.\ \emph{Borel measurable stochastic kernel}) on $Y$ given $X$, if for each $x \in X$, $q(\cdot \,|\,x)$ is a probability measure on $\B(Y)$ and for each $B \in \B(Y)$, $q(B \,|\, \cdot)$ is universally measurable (resp.\ Borel measurable). This definition is equivalent to that the mapping $x \mapsto q(\cdot \,|\, x)$ is measurable from $\big(X, \U(X)\big)$ (resp.\ $\big(X, \B(X)\big)$) into $\big(\P(Y), \B(\P(Y))\big)$; cf.\ \cite[Def.~7.12, Prop.~7.26, Lem.~7.28]{bs}. To refer to the stochastic kernel, we will often use the notation $q$ or $q(dy \,|\, x)$. % rev.

Through out the paper, for summations involving extended real numbers, we adopt the convention $+ \infty - \infty = - \infty + \infty = + \infty$ for technical convenience. (Later our conditions on the MDP model will preclude such summations from occurring in our results.) If $p \in \P(X)$ and $f: X \to [-\infty, + \infty]$ is universally measurable, we define $\int f dp : = \int f^+ dp - \int f^- dp$, where $f^+$ ($f^-$) is the positive (negative) part of $f$ and the integration is w.r.t.\ the completion of $p$. 
For $x \in X$ and $B \subset X$, we denote by $\delta_x$ the Dirac measure concentrating at $x$, by $B^c$ the set $X \setminus B$, and by $\ind_B$ the indicator function for the set $B$. For an event $E$ in a probability space, we write $\ind(E)$ for the indicator of $E$.

\subsection{Borel-Space MDPs with Average Cost Criteria} \label{sec-2.2}

We consider a Borel-space MDP in the universal measurability framework (cf.\ \cite[Chap.\ 8.1]{bs}). Specifically, we assume the following:
\begin{itemize}[leftmargin=0.65cm,labelwidth=!]
\item The state space $\X$ and the action space $\A$ are \emph{Borel spaces}.
\item The control constraint is specified by a set-valued map $A: x \mapsto A(x)$ on $\X$, where $A(x) \subset \A$ is a nonempty set of admissible actions at the state $x$, and the graph of $A(\cdot)$,
$\Gamma := \{(x, a) \mid x \in \X, a \in A(x)\},$
is \emph{analytic}.
\item The one-stage cost function $c: \Gamma \to [-\infty, +\infty]$ is \emph{lower semianalytic}.
\item State transitions are governed by $q(dy \mid x, a)$, a \emph{Borel measurable} stochastic kernel on $\X$ given $\X \times \A$.
\end{itemize}
The control problem has an infinite horizon. 
For $n \geq 0$, let $h_n : = (x_0, a_0, x_1, a_1,\ldots, x_n)$, where $x_n$ and $a_n$ denote the state and action, respectively, at the $n$th stage, and let $\omega : = (x_0, a_0, x_1, a_1, \ldots)$. 
We denote the space of $h_n$ by $H_n : = (\X \times \A)^n \times \X$ and the space of $\omega$ by $\Omega: = (\X \times \A)^\infty$; both spaces are endowed with the product topology so that they are Borel spaces \cite[Prop.\ 7.13]{bs}. 

By a \emph{universally measurable policy} (or a \emph{policy} for short), we mean a sequence of \emph{universally measurable} stochastic kernels, $\pi : =(\mu_0, \mu_1, \ldots)$, where for each $n \geq 0$,
$\mu_n\big(da_n \,|\, h_n \big)$ is a universally measurable stochastic kernel on $\A$ given $H_n$ such that
\begin{equation}  \label{eq-control-constraint}
   \mu_n\big(A(x_n) \mid h_n \big) = 1, \qquad \forall \, h_n = (x_0, a_0, \ldots, a_{n-1}, x_n) \in H_n.
\end{equation}   
Here, for each $h_n$, the set $A(x_n)$ is analytic since it is a section of the analytic set $\Gamma$ (cf.\ \cite[Prop.~7.40]{bs}), and the probability of $A(x_n)$ is measured w.r.t.\ the completion of $\mu_n(da_n \,|\, h_n)$.
If for every $n \geq 0$ and every $h_n \in H_n$, $\mu_n(da_n \,|\, h_n)$ is a Dirac measure, $\pi$ is called a \emph{nonrandomized} policy. A general $\pi \in \Pi$ will sometimes be called randomized or history-dependent, in order to contrast it with nonrandomized policies or policies with more structures.

A \emph{Markov} (resp.\ \emph{semi-Markov}) policy is a policy such that for every $n \geq 0$, as a function of $h_n$, the probability measure $\mu_n(d a_n \,|\, h_n)$ depends only on $x_n$ (resp.\ $(x_0, x_n)$). If a Markov policy $\pi$ has identical stochastic kernels, i.e., $\pi = (\mu, \mu, \ldots)$, we call it a \emph{stationary} policy and write it simply as $\mu$. Likewise, if a semi-Markov policy $\pi$ satisfies that for some stochastic kernel $\tilde \mu$ on $\A$ given $\X^2$, $\mu_0(da \nmid x) = \tilde \mu(da \nmid x, x)$ for all $x \in \X$ and $\mu_n = \tilde \mu$ for all $n \geq 1$, we call $\pi$ a \emph{semi-stationary} policy.

Let $\Pi$ denote the set of all policies, and let $\Pi_m$ ($\Pi_s$) denote the subset of Markov (stationary) policies. 
Since the control constraint $A(\cdot)$ has an analytic graph $\Gamma$, a nonrandomized stationary policy exists by the Jankov-von Neumann selection theorem \cite[Prop.~7.49]{bs}, so $\Pi$, $\Pi_m$, and $\Pi_s$ are all nonempty. By contrast, a Borel measurable policy---a policy consisting of Borel measurable stochastic kernels $\{\mu_n\}$---may not exist \cite{Blk-borel}.

For each policy $\pi = (\mu_0, \mu_1, \ldots) \in \Pi$, an initial state distribution $p_0 \in \P(\X)$ together with the collection of stochastic kernels $\mu_0(d a_0 \nmid x_0)$, $q(dx_1 \nmid x_0, a_0)$, $\mu_1(da_1 \nmid h_1)$, $q(dx_2 \nmid x_1, a_1), \ldots$ determines uniquely a probability measure $\Pr^\pi_{p_0}$ on $\U(\Omega)$ \cite[Prop.\ 7.45]{bs}. If $p_0 = \delta_x$, we will also write $\Pr^\pi_{\delta_x}$ as $\Pr^\pi_x$ and the associated  expectation operator as $\E^\pi_x$. 
Throughout the paper, for notational simplicity, we will write the stochastic process on $\U(\Omega)$ induced by $\pi$ and $p_0$ as $\{(x_n, a_n)\}_{n \geq 0}$, using $(x_n, a_n)$ to denote the random variables $\big(x_n(\omega), a_n(\omega)\big)$ instead, which are the $(x_n, a_n)$-components of $\omega$. 

Let $c^+$ ($c^-$) be the positive (negative) part of the one-stage cost function $c$. 
We define two general classes of MDPs based on the finiteness of finite-stage costs w.r.t.\ $c^+$ or $c^-$. The definition will ensure that for any policy and initial state, the expected $n$-stage cost is well defined and does not involve $+\infty - \infty$. The average cost functions can then be defined properly.

\begin{definition} \label{def-ac-models}
We say an MDP is in the model class \mdp (\mdn), if for $c^\diamond = c^-$ ($c^\diamond = c^+$),
\begin{equation} \label{eq-ac-mdpn}
   \textstyle{ \E^\pi_x \big[ \sum_{k=0}^n c^\diamond(x_k, a_k) \big] < + \infty, \qquad \forall \, x \in \X, \, \pi \in \Pi, \, n \geq 0.}
\end{equation}
\end{definition}
\smallskip

Now consider MDPs in the \mdp or \mdn class. 
For $n \geq 1$, define two $n$-stage cost functions as follows: for $\pi \in \Pi, x \in \X$, and $j \geq 0$,
$$ J_n(\pi, x) : = \E^\pi_x \big[ \, \textstyle{\sum_{k=0}^{n-1} c(x_k, a_k)} \, \big], \qquad J_{n,j}(\pi, x) : = \E^\pi_x \big[ \, \textstyle{\sum_{k=0}^{n-1} c(x_{k+j}, a_{k+j})} \, \big].$$
($J_n$ is the standard $n$-stage cost function; $J_{n,j}$ instead measures the expected costs incurred from time $j$ to time $j+n-1$.)
We consider four different long-run average cost functions: for $\pi \in \Pi, x \in \X$, 
\begin{equation} \label{eq-crt1}
 J^{(1)}(\pi, x) : =   \limsup_{n \to \infty} n^{-1} J_n(\pi, x), \qquad J^{(2)}(\pi, x) : =   \liminf_{n \to \infty} n^{-1}  J_n(\pi, x),
\end{equation}
\begin{equation} \label{eq-crt3}
 J^{(3)}(\pi, x) : =   \limsup_{n \to \infty} \sup_{j \geq 0} n^{-1} J_{n,j}(\pi, x), \qquad J^{(4)}(\pi, x) : =   \liminf_{n \to \infty} \inf_{j \geq 0} n^{-1} J_{n,j}(\pi, x).
\end{equation} 
The optimal average cost functions corresponding to these criteria are given by 
$$ g^*_i(x) : = \inf_{\pi \in \Pi}  J^{(i)}(\pi, x), \qquad 1 \leq i \leq 4.$$
With respect to the $i$th criterion, a policy $\pi \in \Pi$ is called \emph{optimal for state $x$}, if $J^{(i)}(\pi, x) = g^*_i(x)$; 
and \emph{$\epsilon$-optimal for state $x$}, if $\epsilon > 0$ and 
\begin{equation} \label{def-pol-opt}
     J^{(i)}(\pi, x) \leq \begin{cases} 
                g^*_i(x) + \epsilon & \text{if} \ g^*_i(x) > - \infty, \\
                - \epsilon^{-1} & \text{if} \ g^*_i(x) = - \infty.
                \end{cases}   
\end{equation}                
A policy is called \emph{($\epsilon$-)optimal}, if it is ($\epsilon$-)optimal for \emph{all} states $x \in \X$.
            
As can be seen, the definitions of \mdp and \mdn ensure that the expected costs $J_n(\pi, x)$ and $J_{n,j}(\pi, x)$ are well defined and hence the average cost functions $J^{(i)}$ and $g^*_i$ are all well defined. All these functions are extended real-valued: the ranges of $J_n$ and $J_{n,j}$ are contained in $(- \infty, + \infty]$ in the case of \mdp and $[- \infty, + \infty)$ in the case of \mdn, whereas $J^{(i)}$ and $g^*_i$ can take both values $-\infty, + \infty$ in either case.
It can also be seen that $J^{(4)} \leq J^{(2)} \leq J^{(1)} \leq J^{(3)}$, so 
$g^*_4 \leq   g^*_2 \leq g^*_1 \leq g^*_3.$
The two limsup optimal average cost functions, $g^*_1$, $g^*_3$, will be the focus of our study later.

\smallskip
\begin{rem} \rm \label{rmk-ac-cri}
To our knowledge, the average cost criteria $J^{(3)}$ and $J^{(4)}$ have not been considered before in the MDP literature. They are related to the maximal and minimal values of Banach limits (a type of positive linear functional on $\ell_\infty$, the space of bounded sequences of real numbers endowed with the norm $\|\cdot\|_\infty$) \cite[Chap.~3.4]{Kre85}. If the one-stage cost function $c(\cdot)$ is bounded above or below, one actually has 
$$ J^{(3)}(\pi, x) =   \lim_{n \to \infty} \sup_{j \geq 0} n^{-1} J_{n,j}(\pi, x), \qquad J^{(4)}(\pi, x) =   \lim_{n \to \infty} \inf_{j \geq 0} n^{-1} J_{n,j}(\pi, x),$$
where the existence of the limits follows from essentially the same arguments given in the proof of \cite[Theorem 3.4.1]{Kre85}. 
The criterion $J^{(3)}$ seems a useful alternative to the commonly used criterion $J^{(1)}$, as it provides a stronger sense of optimality  
for average-cost MDPs.  
\qed
\end{rem}

\begin{example} \label{ex-uc}
If $c(\cdot)$ is bounded above (below), the MDP is in the class \mdn (\mdp).
Consider now some cases where $c^+, c^-$ are both unbounded. First, suppose that there exist a universally measurable function $w: \X \to [0, + \infty)$ and constants $b \geq 0$, $\beta \in [0,1)$ such that for all $x \in \X$,
\begin{equation} \label{eq-exuc1}
  \sup_{a \in A(x)} \int_\X w(y) \, q(dy \mid x, a) \leq \beta w(x) + b \qquad \text{and} \qquad \sup_{a \in A(x)} c^+(x, a) \leq w(x).
\end{equation}  
Let $\tilde b : = b/(1 - \beta)$, $\tilde w(x) : = w(x)/(1-\beta)$. 
Then for any $\pi \in \Pi$, $x \in \X$, and $k \geq 0$, we have $\E^\pi_x [ c^+(x_k, a_k) ] \leq \tilde b + \lambda^k w(x)$, 
so $ \E^\pi_x \big[ \sum_{k=0}^{n-1} c^+(x_k, a_k) \big] \leq n \tilde b + \tilde w(x) < + \infty$ for all $n \geq 1$ and the MDP is in the class \mdn. 
Moreover, in this case, $n^{-1} J_n(\pi, x) \leq \tilde b + n^{-1} \tilde w(x)$ and $n^{-1} J_{n,j}(\pi, x) \leq \tilde b + n^{-1} \lambda^j \tilde w(x)$ for all $n \geq 1$ and $j \geq 0$. 
So $g^*_i$, $1 \leq i \leq 4$, are bounded above by the constant $\tilde b$.
        
Similarly, if the conditions in (\ref{eq-exuc1}) hold with $c^-$ replacing $c^+$ in the second relation, then the MDP is in the class \mdp. 
Moreover, since in this case $n^{-1}J_n(\pi, x) \geq - \tilde b - n^{-1} \tilde w(x)$ and $n^{-1} J_{n,j}(\pi, x) \geq - \tilde b - n^{-1} \lambda^j \tilde w(x)$, the optimal cost functions $g^*_i$, $1 \leq i \leq 4$, are bounded below by the constant $-\tilde b$. 

A special case of the above is when (\ref{eq-exuc1}) holds with $w(\cdot) \geq 1$ and with the function $|c|(\cdot)$ in place of $c^+$ in the second relation. Then the MDP belongs to both \mdn and \mdp. This class of MDPs has been studied in the literature, under additional assumptions and via different approaches than the one we take in this paper; see e.g., \cite[Chap.~10]{HL99} and \cite{JaN06,VAm03,Yu20}.  \qed
\end{example}
\smallskip

If $c(\cdot)$ is bounded above or below, average costs along sample paths can also be defined. 

\begin{definition} \label{def-ptac-models}
We say an MDP is in the model class \ptmdp (\ptmdn), if the one-stage cost function $c$ is real-valued and bounded below (above) on $\Gamma$.
\end{definition}
\smallskip

For MDPs in \ptmdp and \ptmdn, the following average cost criteria are well defined: 
\begin{align}
   \tJ^{(1)}(\pi, x) & : = \E^\pi_x \left[ \limsup_{n \to \infty} \tilde c_n
   \right],  & \tJ^{(2)}(\pi, x) & : = \E^\pi_x \left[ \liminf_{n \to \infty} \tilde c_n  \right],    \label{eq-pt-ac1} \\
     \tJ^{(3)}(\pi, x)  & : = \E^\pi_x \left[ \lim_{n \to \infty} \sup_{j \geq 0} \tilde c_{n,j} \right],  &
       \tJ^{(4)}(\pi, x)  & : = \E^\pi_x \left[ \lim_{n \to \infty} \inf_{j \geq 0} \tilde c_{n,j} \right], 
      \label{eq-pt-ac2a}  
\end{align}
where 
$$ \textstyle{\tilde c_n : =  n^{-1} \sum_{k=0}^{n-1} c(x_k, a_k), \qquad \ \ \ \tilde c_{n,j} : = n^{-1} \sum_{k=0}^{n-1} c(x_{k+j}, a_{k+j}).}$$
The limits inside the expectations in (\ref{eq-pt-ac2a}) exist by the proof of \cite[Theorem 3.4.1]{Kre85}, as discussed in Remark~\ref{rmk-ac-cri}. 
Under these criteria, the optimal average cost functions $\tg^*_i$ and optimal policies are defined in the same way as in the case of the criteria $J^{(i)}$. In particular, 
$$ \tg^*_i(x) : = \inf_{\pi \in \Pi}  \tJ^{(i)}(\pi, x), \qquad x \in \X, \ 1 \leq i \leq 4,$$
and clearly, $\tg^*_4 \leq \tg^*_2 \leq \tg^*_1 \leq \tg^*_3$. 

For all the criteria introduced above, besides the optimal average cost functions, we will also consider the minimum average costs achievable by Markov policies in $\Pi_m$ or stationary policies in $\Pi_s$. Specifically, 
for $1 \leq i \leq 4$ and with $\diamond$ representing a symbol in $\{m, s\}$, define
\begin{align*}
  g^\diamond_i(x)  & : = \inf_{\pi \in \Pi_\diamond}  J^{(i)}(\pi, x),  
   &  \tg^\diamond_i(x) & : = \inf_{\pi \in \Pi_\diamond} \tJ^{(i)}(\pi, x),  \qquad x \in \X.   
\end{align*}   
We will refer to these functions as \emph{the optimal average cost functions w.r.t.\ $\Pi_\diamond$}. It is a fact that $g^m_i = g^*_i$ (cf.\ the proof of \cite[Prop.~1]{ShrB79}).
For $\epsilon \geq 0$, if a policy $\pi \in \Pi$ (not necessarily Markov or stationary) is such that for all $x \in \X$,
$$\tJ^{(i)}(\pi, x) \leq \begin{cases} 
                \tg^\diamond_i(x) + \epsilon & \text{if} \ \tg^\diamond_i(x) > - \infty, \\
                - \epsilon^{-1} & \text{if} \ \tg^\diamond_i(x) = - \infty,
                \end{cases}   
$$
we say $\pi$ \emph{attains $\tg^\diamond_i$ within $\epsilon$ accuracy}. Similarly, we say $\pi$ attains $g^s_i$ within $\epsilon$ accuracy if an inequality like the above holds for $J^{(i)}(\pi, \cdot)$ and $g^s_i$. In the case $\epsilon = 0$, we also say $\pi$ attains $\tg^\diamond_i$ or $g^\diamond_i$. If an inequality like the above holds for a particular state $x$ instead of all states, $\pi$ will be said to attain $\tg^\diamond_i(x)$ or $g^\diamond_i(x)$ within $\epsilon$ accuracy. 

\smallskip
\begin{rem} \rm 
For a finite-state MDP with an arbitrary action space and a bounded one-stage cost function $c(\cdot)$, these functions are identical: $g^*_i$, $g^s_i$, and $\tg^\star_i$, $\star \in \{\ast, m, s\}$, $i = 1, 2$ (cf.\ \cite{Bie87, Fei80} and the related work \cite[Chap.\ 7.13]{DyY79}). In general, they can be all different when the state space is infinite.
\qed
\end{rem}

\section{Main Results} \label{sec-results}

We first give several basic optimality results in Section~\ref{sec-basic-opt}. 
We then study, in Section~\ref{sec-ae-const}, the almost-everywhere constancy of the optimal average cost functions for the \mdn and \ptmdn models. The results of these two subsections will be proved in Sections~\ref{sec-proofs-basic} and~\ref{sec-5}, respectively.

In what follows, we will place the MDP class label(s) at the start of a theorem to indicate which class(es) of MDPs the theorem is concerned with. 

\subsection{Basic Optimality Results} \label{sec-basic-opt}

To study the average cost problems under the various criteria given in Section~\ref{sec-2.2}, 
we consider the probability measures induced on $\U(\Omega)$ by the policies (also known as \emph{strategic measures} in the literature \cite[Chap.~3.5]{DyY79}).
We shall work with their restrictions to $\B(\Omega)$, 
in order to make use of well-studied properties of Borel probability measures---since a probability measure on $\U(\Omega)$ is uniquely determined by its restriction to $\B(\Omega)$ (cf.\ Lemma~\ref{lem-umset} and Remark~\ref{rmk-lem-umset} about completion of measures), the two are effectively the same. Let
\begin{align*}
  \S  & : = \big\{ p \in \P(\Omega) \mid p = \text{restriction of $\Pr_{p_0}^\pi$ to $\B(\Omega)$}, \, \pi \in \Pi, \, p_0 \in \B(\X)\big\}, \\
\S^0 & : = \big\{ p \in \P(\Omega) \mid p = \text{restriction of $\Pr_{\delta_x}^\pi$ to $\B(\Omega)$}, \,  \pi \in \Pi, \, x \in \X\big\}.
\end{align*}
Define $\S_m, \S^0_m$ as above with $\Pi_m$ in place of $\Pi$, and define $\S_s, \S^0_s$ similarly with $\Pi_s$ in place of $\Pi$.
All these sets are subsets of the Borel space $\P(\Omega)$. 
The next theorem is our first result.

\begin{theorem} \label{thm-strat-m}
For $\S_\star \in \{\S, \S_m, \S_{s}\}$, $\S_\star$ and $\S_\star^0$ are analytic.
\end{theorem}

Since the average cost $J^{(i)}(\pi, x)$ or $\tJ^{(i)}(\pi, x)$ is a function of $\Pr_x^\pi$, we can rewrite the average cost problems under these criteria as optimization problems on the set of probability measures induced by the policies in $\Pi, \Pi_m$, or $\Pi_s$. 
Theorem~\ref{thm-strat-m} allows us to cast these optimization problems as partial minimization problems involving lower semianalytic functions---a class of problems whose optimality properties are well studied (cf.\ Section~\ref{sec-proof-review}). We then transfer the results from the space of induced probability measures to the space of policies, thereby obtaining the optimality results given below.

\begin{theorem}[\mdp, \mdn] \label{thm-ac-basic}
For any average cost criterion $J^{(i)}$, $1 \leq i \leq 4$, defined in (\ref{eq-crt1})-(\ref{eq-crt3}), the following hold:
\begin{enumerate}[leftmargin=0.7cm,labelwidth=!]
\item[\rm (i)] The functions $g^*_i=g^m_i$ and $g^s_i$ are lower semianalytic.
\item[\rm (ii)] For every $\epsilon > 0$, there exists a randomized semi-Markov $\epsilon$-optimal policy. If there exists an optimal policy for each state $x \in \X$, then there exists a randomized semi-Markov optimal policy.
\item[\rm (iii)] For every $\epsilon > 0$, there exists a randomized semi-stationary policy that attains $g^s_i$ within $\epsilon$ accuracy. This also holds for $\epsilon=0$ if $g^s_i(x)$ is attained by some stationary policy for each state $x \in \X$. 
\end{enumerate}
\end{theorem}

\begin{theorem}[\ptmdp, \ptmdn]  \label{thm-ac-basic2}
For any average cost criterion $\tJ^{(i)}$, $1 \leq i \leq 4$, defined in (\ref{eq-pt-ac1})-(\ref{eq-pt-ac2a}), the following hold:
\begin{enumerate}[leftmargin=0.7cm,labelwidth=!]
\item[\rm (i)] The functions $g^*_i, g^m_i$, and $g^s_i$ are lower semianalytic. 
\item[\rm (ii)] For every $\epsilon > 0$, there exists a randomized (history-dependent) $\epsilon$-optimal policy. If there exists an optimal policy for each state $x \in \X$, then there exists a randomized optimal policy.
\item[\rm (iii)] For every $\epsilon > 0$, there exists a randomized semi-Markov (resp.\ semi-stationary) policy that attains $\tg^m_i$ (resp.\ $\tg^s_i$) within $\epsilon$ accuracy. This also holds for $\epsilon=0$ if $\tg^m_i(x)$ (resp.\ $\tg^s_i(x)$) is attained by some Markov (resp.\ stationary) policy for each state $x \in \X$.
\end{enumerate}
\end{theorem}

\begin{rem}[comparison with some prior results on average-cost MDPs] \label{rmk-basicthm-compare} \hfill \\
\noindent (a) As pointed out by Feinberg~\cite{Fei80}, Strauch's results \cite[Lem.~4.1 and the proof of Thm.~8.1]{Str-negative} for discounted and total cost Borel-space MDPs can be carried over to the average cost case. In particular, if the set $\Gamma$ of feasible state-action pairs is Borel and a Borel measurable policy exists, 
under any average cost criterion $J^{(i)}$ (resp.\ $\tJ^{(i)}$), one can extend Strauch's arguments to show, for any $\rho \in \P(\X)$ and $\epsilon > 0$, the existence of a randomized Borel measurable semi-Markov (resp.\ history-dependent) policy that is $\epsilon$-optimal for \emph{$\rho$-almost all} states. This $\rho$-almost-everywhere $\epsilon$-optimality is due to the restriction of the policy space to include only Borel measurable policies. By contrast, with universally measurable policies, there exist policies that are optimal or $\epsilon$-optimal \emph{everywhere}.\smallskip

\noindent (b) Under the standard average cost criterion $J^{(1)}$, it is known that even in an MDP with a countable state space, a finite action space, and bounded one-stage costs, there need not exist a nonrandomized semi-Markov $\epsilon$-optimal policy \cite[Example 3, Chap.~7]{DyY79} nor a randomized Markov $\epsilon$-optimal policy \cite[Sec.~5]{Fei80}. In both counterexamples, there exists an optimal policy for each state. So without extra conditions on the MDP, under $J^{(1)}$, Theorem~\ref{thm-ac-basic}(ii) is the strongest possible.\smallskip

\noindent (c) For reward (cost) criteria that are convex (concave) w.r.t.\ the probability measure induced on $\Omega$, Feinberg \cite{Fei82,Fei82a} proved the existence of \emph{nonrandomized} policies that are $\epsilon$-optimal, in the $\rho$-almost-everywhere sense mentioned earlier, for MDPs with Borel measurable policies. By extending his analyses to universally measurable policies and combining the results with Theorems~\ref{thm-ac-basic}-\ref{thm-ac-basic2}, we can show the following. When the one-stage cost function is bounded below, under the criteria $J^{(2)}$, $J^{(4)}$, and $\tJ^{(i)}$, $1 \leq i \leq 4$,  the conclusions of Theorem~\ref{thm-ac-basic}(ii) and Theorem~\ref{thm-ac-basic2}(ii) still hold if we replace ``randomized'' by ``nonrandomized'' and replace ($\epsilon$-)optimality by ($\epsilon$-)optimality in the $\rho$-almost-everywhere sense, and similarly, the conclusions of Theorem~\ref{thm-ac-basic2}(iii) for the case $\tilde g_i^m$ still hold if we replace ``a randomized semi-Markov policy'' that attains $\tilde g_i^m(x)$ within $\epsilon$-accuracy for all states $x$ 
by ``a nonrandomized semi-Markov policy'' that attains $\tilde g_i^m(x)$ within $\epsilon$-accuracy for $\rho$-almost all states $x$, where $\rho$ is any given probability measure on $\X$. 

Notice the $\rho$-almost-everywhere sense of optimality in the preceding statements. It arises from a final step in the analysis, which uses the Blackwell--Ryll-Nardzewski selection theorem \cite[Thm.~2]{BlR63} to extract a desired nonrandomized policy from, in general, an uncountable set of nonrandomized policies, the mixture of which, roughly speaking, represents a randomized policy that is {($\epsilon$-)optimal} everywhere. It is an open question whether an alternative selection theorem can be used in this analysis to strengthen the $\rho$-almost-everywhere optimality to optimality. 
\qed
\end{rem}

%\smallskip
\begin{rem}[about the proofs] \label{rmk-basic-thm-prf}
(a) The proof of $\S$ and $\S^0$ being analytic sets is similar to the proofs given in \cite[Sec.~7]{Str-negative} and \cite[Chap.\ 3.5]{DyY79} for the strategic measures induced by Borel measurable policies, except that the latter analyses deal with Borel sets, whereas the proof here also relies critically on the properties of analytic sets and of probability measures on such sets. \smallskip

\noindent (b) For the expected average cost criteria $J^{(i)}$, Theorem~\ref{thm-ac-basic} can be proved by working with the marginal probability distributions of $(x_n, a_n)$, $n \geq 0$, instead of the strategic measures, induced by the policies. This proof approach largely follows the one used by Shreve and Bertsekas \cite{ShrB79} (see also the book \cite[Chap.\ 9]{bs}) to analyze discounted/total cost MDPs with universally measurable policies.
They called those marginal distributions \emph{admissible sequences}, showed that the set $\Delta$ of all such sequences is
analytic \cite[Lem.~1]{ShrB79}, and related a discounted or total cost MDP to a partial minimization problem on $\Delta$. 
A proof of Theorem~\ref{thm-ac-basic}(i)-(ii) using this approach%footnote starts
\footnote{The proof is not exactly the same as Shreve and Bertsekas's proof mentioned above. The main difference is the following: they related the partial minimization problem on $\Delta$ to a deterministic control problem and transferred the optimality equations and other optimality properties of the latter problem to the original discounted/total cost MDP. In the average cost case, like the original MDP, that deterministic control problem need not possess useful DP-type properties, so we instead work directly with the partial minimization problem and its ($\epsilon$-)optimal solution mappings.}
%footnote ends 
can be found in the author's report \cite[Appendix A]{Yu19} (that proof is for the criterion $J^{(1)}$ and two special cases of the \mdp and \mdn models; however, the arguments for the more general cases are essentially the same). 

Although we do not take this proof approach in the present paper, our proofs still involve those induced marginal probability distributions: among others, we use the analyticity of the set $\Delta$ in proving Theorem~\ref{thm-strat-m} for the subsets $\S_m$ and $\S_s$ of strategic measures.\smallskip

\noindent (c) The advantage of working with the strategic measures is that it allows one to handle the two types of criteria, $J^{(i)}$ and $\tJ^{(i)}$, in the same way. Moreover, the analysis carries over to other risk-sensitive criteria, such as the average criteria considered by \cite{CaS10,Jas07} for risk-sensitive MDPs. In particular, Theorem~\ref{thm-strat-m} and the proof arguments of Theorems~\ref{thm-ac-basic}-\ref{thm-ac-basic2} are applicable to risk-sensitive MDPs with Borel spaces and universally measurable policies.
\qed
\end{rem}

\subsection{Almost-Everywhere Constancy of Optimal Cost Functions} \label{sec-ae-const}

We now study the structure of the optimal average cost functions for the \mdn and \ptmdn models.

\subsubsection{Results under a General Reachability Condition} \label{sec-ae-const-a}

Consider first the two limsup optimal average cost functions $g^*_1$ and $g^*_3$ for the \mdn model. 
We introduce additional assumptions to bound from above the expected one-stage costs over time.
Let $\M_+(\X)$ denote the set of all real-valued, nonnegative, universally measurable functions on $\X$.

\begin{assumption} \label{cond-ac-n1}
For each $\pi \in \Pi$, $M_\pi(x) : = \sup_{n \geq 0} \E_x^\pi [ c^+(x_n, a_n) ] < \infty$ for all $x \in \X$, and
$\int_\X M_\pi(y) \, q(dy \,|\, x, a) < \infty$ for all $(x, a) \in \Gamma$.
\end{assumption}

\begin{assumption} \label{cond-ac-n2}
There exists a function $M \in \M_+(\X)$ such that: 
\begin{enumerate}[leftmargin=0.65cm,labelwidth=!]
\item[\rm (a)] for each $x \in \X$, $\sup_{n \geq 0} \E_x^\pi [ c^+(x_n, a_n) ] \leq  M(x)$ for some policy $\pi \in \Pi$;
\item[\rm (b)] $\sup_{n \geq 0} \E_x^\pi [ M(x_n) ] <  \infty$ for all $x \in \X, \pi \in \Pi$. 
\end{enumerate}
\end{assumption}

\begin{rem} \rm
Assumption~\ref{cond-ac-n1} or \ref{cond-ac-n2} implies that $g^*_i < + \infty$, $i = 1, 3$. Assumptions~\ref{cond-ac-n1} and \ref{cond-ac-n2} are clearly satisfied, if $c(\cdot)$ is bounded above or if (\ref{eq-exuc1}) in Example~\ref{ex-uc} holds. Specifically, in the latter case, with $w(\cdot)$, $\beta$, $b$, and $\tilde b$ as in Example~\ref{ex-uc}, we have $M_\pi(x) \leq \tilde b + w(x)$, $\int_\X M_\pi(y) \, q(dy \,|\, x, a) \leq b + \tilde b + \beta w(x)$, and for the function $M(x) : = \tilde b + w(x)$, we have $\sup_{n \geq 0} \E_x^\pi [ M(x_n) ] \leq 2 \tilde b + w(x)$.\qed
\end{rem} 
 
\begin{lem}[\mdn] \label{lem-ineq}
Under Assumption~\ref{cond-ac-n1}, for $g^*=g^*_1$ or $g^*_3$, we have
\begin{equation} \label{eq-ac-ineq}
 g^* (x) \leq \inf_{a \in A(x)} \int_\X g^*(y) \, q(dy \mid x, a), \qquad \forall \, x \in \X.
\end{equation} 
\end{lem} 

Note that by Theorem~\ref{thm-ac-basic}(i) and Assumption~\ref{cond-ac-n1}, the integrals in (\ref{eq-ac-ineq}) are well defined. 

The inequality (\ref{eq-ac-ineq}) implies that with proper regularity conditions in place, under any policy, $g^*(x_n)$, $n \geq 0$, form a submartingale sequence. By making use of convergence and optional stopping theorems for submartingales, we then obtain the following theorem about the almost-everywhere constancy of $g^*_1$ and $g^*_3$. 

Regarding notation used below, for a process $\{x_n\}$ and a set $B \subset \X$, $\tau_B : = \min \{ n \geq 0 \mid x_n \in B\}$ (the first hitting time of $B$); by convention $\tau_B : = +\infty$ if $\{n \geq 0 \mid x_n \in B\} = \varnothing$. If $s$ and $t$ are two extended real numbers, $s \wedge t : = \min \{s, t\}$. The abbreviations ``a.e.''~and ``a.s.''~stand for ``almost everywhere'' and ``almost surely,'' respectively. 

\begin{theorem}[\mdn] \label{thm-ac-const}
Let Assumptions~\ref{cond-ac-n1}-\ref{cond-ac-n2} hold. Let $g^*=g^*_1$ or $g^*_3$. 
Suppose that there exist a nontrivial $\sigma$-finite measure $\lambda$ on $\U(\X)$ and a set $\hat \X \in \U(\X)$ with $\lambda({\hat \X}^c) = 0$ such that:
\begin{enumerate}[leftmargin=0.65cm,labelwidth=!]
\item[\rm (i)] for each Borel set $B \subset \hat \X$ with $\lambda(B) > 0$ and each $x \in \hat \X$, 
$$  \Pr_x^{\pi_B} ( \tau_B < \infty) = 1 \quad \text{for some} \ \pi_B \in \Pi;$$ 
\item[\rm (ii)] $g^* \not \equiv - \infty$ on $\hat \X$.
\end{enumerate}
Then the following hold:
\begin{enumerate}[leftmargin=0.65cm,labelwidth=!] 
\item[\rm (a)] There is a finite constant $\ell_\lambda$ such that $g^* = \ell_\lambda$ $\lambda$-a.e. 
\item[\rm (b)] 
Let $D = \{y \in \hat \X \mid g^*(y) = \ell_\lambda \}$ and let $f \in \M_+(\X)$ with $f \geq g^*$.
For any initial state $x \in \X$ such that there is a policy $\pi \in \Pi$ under which $\Pr_x^{\pi} (\tau_D < \infty) = 1$ and $\{f(x_n) \}_{n \geq 1}$ are uniformly integrable, we have 
$$g^*(x_{n \wedge \tau_D}) \leq \ell_\lambda \ \ \ \forall \, n \geq 0, \ \ \ \text{$\Pr_x^{\pi}$-a.s.},$$
so, in particular, $g^*(x) \leq \ell_\lambda$. Thus, if $g^*$ is bounded above, then $g^* \leq \ell_\lambda$ on $\hat \X$.
\end{enumerate}
\end{theorem}

\begin{rem} \rm \label{rmk-1}
For an MDP, there can be multiple $\sigma$-finite measures $\lambda$ satisfying condition (i) of Theorem~\ref{thm-ac-const} and yielding the conclusions $g^* = \ell_\lambda$ $\lambda$-a.e.\ on different subsets of state space and for possibly different constants $\ell_\lambda$. 
But if two such measures $\lambda_1$ and $\lambda_2$ are not mutually singular, 
then clearly $\ell_{\lambda_1} = \ell_{\lambda_2}$.
It is also obvious from (\ref{eq-ac-ineq}) that if $g^* = \ell_\lambda$ $\lambda$-a.e., then $g^*(x) \leq \ell_{\lambda}$ for all $x$ such that $q(dy \,|\, x ,a)$ is absolutely continuous w.r.t.\ $\lambda$ for some action $a \in A(x)$. 
\qed
\end{rem}

\begin{rem} \rm \label{rmk-conn-cls}
If $\X$ and $\A$ are finite, the reachability condition given in condition (i) of Theorem~\ref{thm-ac-const} is closely related to the notion of a \emph{connected class} introduced by Platzman \cite[Def.\ 2]{Pla77} for a finite state and action MDP.
A connected class is a subset $C$ of states such that for any $x, x' \in C$, it is possible to reach state $x'$ from state $x$ under some policy, whereas no states outside $C$ can be reached from a state in $C$. This implies that, with $\hat \X = C$ and with $\lambda$ having $C$ as its support, condition (i) holds. Conversely, if condition (i) holds, then the support of $\lambda$ is a subset of some connected class. On a connected class, $g^*_1$ is constant \cite[Thm.~2]{Pla77}; if $\X$ consists of a single connected class plus states that are transient under all policies (i.e., the \emph{weakly communicating} case), then $g^*_1$ is constant \cite{Pla77}. These well-known results for finite-space MDPs are consistent with the conclusion of Theorem~\ref{thm-ac-const}. 

However, these prior results just mentioned rely on the fact that in a finite state and action MDP, if $\Pr^\pi_x(\tau_B < \infty) = 1$, then $\E^{\pi}_x[\tau_B] < \infty$, which is in general not true for infinite-space MDPs (cf.\ Example~\ref{ex-const-1} in Section~\ref{sec-3.2.3}).
Thus, despite similarities, there are essential differences between condition (i) of Theorem~\ref{thm-ac-const} for infinite-space MDPs and the notion of a connected class for finite-space MDPs. 
\qed
\end{rem}

The counterpart of Theorem~\ref{thm-ac-const} for the average cost criteria $\tJ^{(i)}$, $1 \leq i \leq 4$, is given below. As can be verified directly by using Theorem~\ref{thm-ac-basic2}, for the \ptmdn and \ptmdp models, and w.r.t.\ any criterion $\tJ^{(i)}$, we have the equality
\begin{equation} \label{eq-ac2-eq}
 \tg^*_i (x) = \inf_{a \in A(x)} \int_\X \tg^*_i(y) \, q(dy \mid x, a), \qquad \forall \, x \in \X;
\end{equation} 
and under additional assumptions, we can have the function $\tg^m_i$ satisfy an inequality like (\ref{eq-ac-ineq}). 
Recall also that in the case of \ptmdn, $\tg^*_i$ and $\tg^m_i$ are bounded from above. The same proof for Theorem~\ref{thm-ac-const} then leads to the following conclusions.

\begin{theorem}[\ptmdn] \label{thm-ac-const2}
For $1 \leq i \leq 4$, let $\tg^\star := \tg^*_i$; or let $\tg^\star := \tg^m_i$ and assume that for every $\epsilon > 0$, there is a Markov policy $\pi \in \Pi_m$ that attains $\tg^m_i$ within $\epsilon$ accuracy.
In either case, if condition (i) of Theorem~\ref{thm-ac-const} holds and $\tg^\star \not\equiv - \infty$ on $\hat \X$, then there is a finite constant $\ell_\lambda$ such that $\tg^\star = \ell_\lambda$ $\lambda$-a.e.\ and $\tg^\star \leq \ell_\lambda$ on $\hat \X$.
\end{theorem}

\begin{rem} %\rm
(a) For the case $\tg^\star = \tg^*_i$, instead of using submartingale-based arguments, one can prove the preceding theorem directly, by considering those policies that act in a nearly optimal way once the system reaches a set $B$ involved in the assumption. The proof is a straightforward combination of the almost sure finiteness of $\tau_B$ with the sample-path based definition of the average cost criterion $\tJ^{(i)}$.
Such a direct proof does not work for the case $\tg^\star = \tg^m_i$, however. The reason is that $\tg^m_i$ is defined over Markov policies, whereas $\tau_B$ is a random time and a policy that acts in a certain manner from the time $\tau_B$ onwards is, in general, not Markov but history-dependent. \smallskip

\noindent (b) 
If the functions $g^s_i$ and $\tg^s_i$ satisfy an inequality like (\ref{eq-ac-ineq}), then the results of this subsection can be applied to them as well.
\qed
\end{rem}

\subsubsection{Results under a Recurrence Condition involving a Stationary Policy}

Sometimes it would be convenient to have a single stationary policy $\mu$ fulfill the roles of all the policies $\pi_B$ required by condition (i) of Theorem~\ref{thm-ac-const}. The rich theory of Markov chains can be brought to bear to study this special case, and the purpose of the present subsection is to demonstrate this point. Our discussion below will involve some concepts and standard terminology for Markov chains; for their definitions, we refer the reader to Appendix~\ref{appsec-mc} and the references therein.

We start with a lemma that is important for the study of universally measurable stationary policies in general. 
A universally measurable stationary policy $\mu \in \Pi_s$ induces a Markov chain $\{x_n\}$ on $\big(\X, \U(\X)\big)$. If $\X$ is uncountably infinite, the cardinality of $\U(\X)$ is larger than $\B(\X)$ \cite[Appendix~B.5]{bs}, so, unlike $\B(\X)$, $\U(\X)$ is not countably generated. 
Many theorems for irreducible Markov chains were proved for state spaces with countably generated $\sigma$-algebras. 
The following result from Orey~\cite{Ore71} allows us to apply them to Markov chains on $(\X, \U(\X))$.

Recall that the \emph{transition probability function} of a Markov chain on a measure space $(X, \Sigma)$ is a function $P(\cdot \,|\, \cdot)$ defined on $\Sigma \times X$ such that $P(\cdot \,|\, x)$ is a probability measure on $\Sigma$ for each $x \in X$ and $P(E\,|\, \cdot)$ is measurable w.r.t.\ $\Sigma$ for each $E \in \Sigma$.

\begin{lem}[cf.\ {\cite[Prop.\ 1.3]{Ore71}}] \label{lem-sub-sigalg}
Let $P$ be a transition probability function on a measure space $(X, \Sigma)$. Then, for any countable family $\{ E_n\} \subset \Sigma$, there exists a countably generated $\sigma$-algebra $\cE$ such that $\{E_n\} \subset \cE \subset \Sigma$ and for each $E \in \cE$, $P(E \,|\,\cdot)$ is measurable w.r.t.~$\cE$; in other words, $P$ restricted to $\cE \times X$ is a transition probability function on the measure space $(X, \cE)$.
\end{lem}

Based on Lemma~\ref{lem-sub-sigalg}, for a Markov chain $\{x_n\}$ induced by a policy $\mu \in \Pi_s$, we can consider its transition probability function $P_\mu$ restricted to $\cE_\mu(\X) \times \X$, for a countably generated $\sigma$-algebra $\cE_\mu(\X)$ (which depends on $\mu$) such that $\B(\X) \subset \cE_\mu(\X) \subset \U(\X)$. Using this and a convenient fact about completion of measures (cf.\ Lemma~\ref{lem-umset} and Remark~\ref{rmk-lem-umset}), 
we can show that the theory developed for irreducible Markov chains on state spaces with countably generated $\sigma$-algebras can be applied to irreducible Markov chains on $(\X, \U(\X))$ (see Section~\ref{sec-5.2} for details). 

Now, returning to the problem of our interest, we provide a special case of condition (i) of Theorem~\ref{thm-ac-const} in the next lemma, which is derived from the relationship between recurrent Markov chains and Harris recurrent Markov chains. 

\begin{lem} \label{lem-ac-cond-mc}
Let $\mu \in \Pi_s$ and consider the Markov chain $\{x_n\}$ on $\big(\X, \U(\X)\big)$ induced by $\mu$.
Suppose that a set $\tilde \X \in \U(\X)$ is absorbing and indecomposable and,
restricted to $\tilde \X$, $\{x_n\}$ is $\psi$-irreducible and recurrent, where $\psi$ is a maximal irreducibility measure. 
Then, on $\tilde \X$, the restricted Markov chain $\{x_n\}$ has a unique maximal Harris set $\bar H$; condition (i) of Theorem~\ref{thm-ac-const} holds with $\lambda = \psi$ and $\hat \X = \bar H$.
\end{lem}

\begin{prop}[\mdn, \ptmdn; special cases of Theorems~\ref{thm-ac-const}-\ref{thm-ac-const2}] \label{prp-ac-const-mc}
Make the following replacements in the conditions of Theorems~\ref{thm-ac-const}-\ref{thm-ac-const2}: 
\begin{enumerate}[leftmargin=0.65cm,labelwidth=!]
\item[\rm (a)] replace condition (i) of Theorem~\ref{thm-ac-const} with the existence of a stationary policy $\mu$ and a set $\tilde \X$ satisfying the condition of Lemma~\ref{lem-ac-cond-mc}; and 
\item[\rm (b)] replace the condition that $g^*, \tg^\star \not\equiv - \infty$ on $\hat \X$ with the condition that $g^*, \tg^\star \not\equiv - \infty$ on $\tilde \X$. 
\end{enumerate}
Then the conclusions of Theorems~\ref{thm-ac-const}-\ref{thm-ac-const2} 
hold for $\lambda = \psi$ and $\hat \X = \bar H$, where $\psi$ and $\bar H$ are as in Lemma~\ref{lem-ac-cond-mc}. 
\end{prop}

\begin{rem} \rm \label{rmk-mc-related}
(a) In the context of Prop.~\ref{prp-ac-const-mc}, 
the recurrent Markov chain $\{x_n\}$ on $\tilde \X$ has a unique (up to constant multiples) invariant measure $\rho$, which is equivalent to $\psi$ \cite[Cor.~5.2 and Prop.~5.6]{Num84}.
So one of the conclusions of Prop.~\ref{prp-ac-const-mc}
can also be stated as that $g^*$ or $\tg^\star$ is a constant $\rho$-a.e.
\smallskip

\noindent (b) If we assume $g^*$ is bounded above on $\tilde \X$, 
then instead of relying on Theorem~\ref{thm-ac-const}, Prop.~\ref{prp-ac-const-mc} can also be proved directly by using well-known properties of superharmonic functions for Harris recurrent Markov chains. This is true for the case of $\tg^\star$ as well, since for the \ptmdn model, $\tg^\star$ is bounded above. Another related comment is that if a stationary policy $\mu$ induces a Harris recurrent Markov chain $\{x_n\}$ on $\X$ and the following equality holds for a bounded function $g = g^*$ or $\tg^\star$,
$$g(x) =  \inf_{x \in A(x)} \int_\X g(y) \, q(dy \mid x, a) = 
\int_\A \int_\X g(y) \, q(dy \mid x, a) \, \mu(da \mid x), \qquad \forall \, x \in \X,$$ 
then $g^*$ or $\tg^\star$ is a constant, by the well-known fact that for a Harris recurrent Markov chain, a bounded harmonic function must be a constant \cite[Thm.~3.8(i)]{Num84}. (See Section~\ref{sec-5.2} for the details of these comments.) 
\qed
\end{rem}

\begin{rem} \rm \label{rmk-mc-related2}
We now discuss Hopf's decomposition and several other results from Markov chain theory, in connection with the preceding results, to provide a better context for understanding the recurrence condition of Lemma~\ref{lem-ac-cond-mc} and the reachability condition (i) of Theorem~\ref{thm-ac-const} and to help illustrate their broad applicability. 
Consider an arbitrary nontrivial $\sigma$-finite measure $\phi$ on $\U(\X)$ and an arbitrary stationary policy, or more generally, any policy $\pi \in \Pi$ that induces an embedded (time-homogenous) Markov chain on $\X$ at certain stopping times $\tau_m, m \geq 0$. For the resulting Markov chain $\{\hat x_m\}$, let $\hat P$ be its transition probability function. One can construct from $\hat P$ and $\phi$ a $\sigma$-finite measure $\hat \phi$ such that both $\phi$ and $\int \hat P(\cdot \,|\, x) \, \hat \phi(dx)$ are absolutely continuous w.r.t.\ $\hat \phi$ \cite[p.~10]{Num84}. 
With respect to $\hat \phi$ and the Markov chain $\{\hat x_m\}$, the state space $(\X, \U(\X))$ can be partitioned into a dissipative part $\X_d$ and a conservative part $\X_c$, known as Hopf's decomposition \cite[Thm.~3.5]{Num84}: on $\X_d$ the Markov chain is dissipative, whereas $\X_c$ is absorbing and, in the case $\hat \phi(\X_c) > 0$, $\hat \phi$-conservative (see \cite[Def.~3.4]{Num84} for definition). In the latter case, on any absorbing, indecomposable subset $\hat \X_c \subset \X_c$ with $\hat \phi(\hat \X_c) > 0$, the Markov chain is irreducible and recurrent with a maximal irreducibility measure given by the restriction of $\hat \phi$ to $\hat \X_c$ \cite[Thm.~3.6 and Prop.~3.11]{Num84}. This shows that one can use any $\sigma$-finite measure $\phi$ and any stationary or non-stationary policy of the structure discussed above to try to identify subsets of $\X$ on which $g^*$ or $\tg^\star$ is almost everywhere constant.
\qed
\end{rem}

\subsubsection{Illustrative Examples and Further Discussion} \label{sec-3.2.3}

We start with a simple countable-state example. 

\begin{example} \label{ex-const-1}
Consider a countable-state Markov chain example from \cite[Chap.\ 11.1, p.~259]{MeT09} as an uncontrolled MDP: 
the state space $\X = \{0, 1, 2, \ldots\}$; state $0$ is absorbing; and
from state $k \geq 1$, the system moves to state $0$ with probability $\beta_k > 0$ and to state $k+1$ with probability $1 - \beta_k$.
This Markov chain is positive recurrent with $0$ being the only recurrent state.

Let us set $\beta_k = 1/(k+1)$ and define the one-stage costs to be $c(k) = k$ for $k \geq 0$. 
Then $\E_0 [ c(x_n)] = 0$ and $\E_k [ c(x_n)] = \big( \prod_{j=k}^{k+n-1} (1 - \beta_j) \big) \cdot (k+n) = k$ for $k \geq 1$ and $n \geq 0$. 
So Assumptions~\ref{cond-ac-n1}-\ref{cond-ac-n2} are satisfied, and for any criterion $J^{(i)}$, 
the optimal average costs are given by $g^*(k) = k$, $k \geq 0$.

Condition~(i) of Theorem~\ref{thm-ac-const} holds for $\hat \X = \X$ and $\lambda = \delta_0$: indeed, 
from state $k \geq 1$, the probability of never reaching state $0$ is $\prod_{j=k}^\infty (1 - \beta_j) = 0$, since $\sum_{j \geq k} \beta_j = +\infty$.
The Markov chain is, in fact, positive Harris recurrent.  
However, the expected time to reach state $0$ is infinite: with $\tau_0 : = \min \{ n \geq 0 \mid x_n = 0 \}$, we have $\E_k [ \tau_{0} ] = \sum_{n=0}^\infty \prod_{j=k}^{k+n-1} (1 - \beta_j) = \sum_{n=0}^\infty k/(k+n) = + \infty$. This manifests a key difference between condition (i) of Theorem~\ref{thm-ac-const} for infinite-space MDPs and the condition that defines a weakly communicating MDP in the finite state and action case.

We have $\ell_\lambda =  g^*(0) = 0$ and $g^*(k) > \ell_\lambda$ for $k \geq 1$. 
It is not surprising that $g^*(k) \not\leq \ell_\lambda$ on $\hat \X$, since the uniform integrability condition in Theorem~\ref{thm-ac-const}(b) does not hold in this case: for any $t > 0$ and $k \geq 1$, we have 
$$\textstyle{ \sup_{n \geq 0} \E_k \big[ g^*(x_n) \, \ind (g^*(x_n) \geq t) \big] = \sup_{n \geq 0} \big( \prod_{j=k}^{k+n-1} (1 - \beta_j) \big) \cdot (k+n) \, \ind (k+n \geq t) = k,}$$ 
so $\{g^*(x_n)\}$ are not uniformly integrable if the initial state $x_0 = k \geq 1$. 

If $\sup_{k \geq 1} c(k) < + \infty$ instead, then according to Theorem~\ref{thm-ac-const} or Prop.~\ref{prp-ac-const-mc}, $g^*(k) \leq \ell_\lambda = 0$ and hence $g^* \equiv 0$. This is indeed the case, as can be verified by a direct calculation. \qed
\end{example}

Main approaches to study average-cost MDPs on infinite-spaces include the vanishing discount factor approach, the fixed point approach, and the linear programming (LP) approach (see e.g., \cite{HL96,HL99,VAm03,VAm18}). These approaches aim directly at finding stationary optimal policies. Some of the conditions they require are much stronger than those needed in applying our results to study the structure of the optimal average cost functions. 

For example, for the vanishing discount factor approach to work, a family of discounted relative value functions needs to satisfy a certain pointwise boundedness condition, which fails to hold for the MDP in Example~\ref{ex-const-1}, even if we set $c(k) =1$ for $k \geq 1$ so that $g^* \equiv 0$ (for details, see the discussion in \cite[Example 3.1]{Yu-minpair19}). The fixed point approach requires the Markov chains induced by the nonrandomized stationary policies to be uniformly $w$-geometrically ergodic w.r.t.\ a certain weight function $w$ (cf.\ \cite{HL99,VAm03}). By contrast, to apply Prop.~\ref{prp-ac-const-mc}, we only need a recurrent Markov chain induced by some policy on a subset of states, and this Markov chain can be null recurrent (cf.\ the subsequent Example~\ref{ex-const-3}).

Regarding the LP approach for infinite-space MDPs, a large part of the theory relies on the existence of stationary policies that can induce positive recurrent Markov chains (cf.\ \cite{Bor88,HoL94} and \cite[Chap.\ 12]{HL99}), and it cannot handle null recurrent Markov chains. By contrast, for a countable-state MDP with bounded one-stage costs, if there exists a policy inducing a null recurrent Markov chain w.r.t.\ the counting measure, one can immediately conclude that $g^*$ is constant on $\X$. Moreover, even in the positive recurrent case, transient states can be a thorny issue for applying the LP approach, as the example below demonstrates. 

\begin{example}[feasibility issue in linear programs for infinite-space MDPs] \label{ex-const-2}
Consider again the countable-state, finite-action MDP in the previous example; the values of $c(k)$ are unimportant for this discussion.
Hordijk and Lasserre \cite{HoL94} showed that under suitable conditions, a stationary average-cost optimal policy can be constructed from an optimal solution of a linear program. For the MDP in our example, that linear program involves two variables $\gamma$ and $\nu$ which take values in the space of \emph{finite measures} on $\X$ and must satisfy the following linear constraints:
\begin{align}
 \gamma (k) - \textstyle{\sum_{j \geq 0}} \, p_{jk} \gamma(j) & = 0, \qquad \forall \, k \geq 0,  \label{eq-lp1} \\
   \gamma(k) + \nu(k) - \textstyle{ \sum_{j \geq 0}} \, p_{jk} \nu(j) & = b_k, \quad \ \, \forall \, k \geq 0,  \label{eq-lp2}
\end{align}
where $p_{jk}$ is the probability of transitioning from state $j$ to $k$, and the constants $b_k > 0$, $k \geq 0$, satisfy $\sum_{k \geq 0} b_k = 1$.
(Roughly speaking, the first (second) constraint deals with the recurrent (transient) states in a Markov chain associated with a solution.) 
If a pair $(\gamma, \nu)$ is a feasible solution, $\gamma$ must be an invariant probability measure (cf.\ \cite{HoL94}) and hence, for our case, $\gamma(0)=1, \gamma(k) = 0$, $k \geq 1$. Equation (\ref{eq-lp2}) then becomes
$$  \textstyle{\sum_{j \geq 1}} \, \beta_j \, \nu(j) = 1 - b_0, \qquad \nu(k) - (1 - \beta_{k-1})  \, \nu(k-1) = b_k, \quad \forall \, k \geq 1.$$
Since the $b_k$'s are positive, the second relation above implies that $\nu(1) > 0$ and for all $k \geq 1$, 
$$\nu(k) > 
(1 - \beta_{k-1}) \, \nu({k-1}) > (1 - \beta_{k-1}) \, (1 - \beta_{k-2})  \, \nu(k-2) > \cdots > \big(\textstyle{\prod_{j=1}^{k-1} } (1 - \beta_j) \big) \cdot \nu(1).$$
With $\beta_j = 1/(j+1)$, we then have
$$\textstyle{\sum_{k \geq 1} \nu(k) \geq  \nu(1) \cdot \sum_{k \geq 1} \prod_{j=1}^{k-1} (1 - \beta_j)  =  \nu(1) \cdot \sum_{k \geq 1} 1/k = + \infty,}$$
so $\nu$ is not a finite measure. This shows that the constraints (\ref{eq-lp1})-(\ref{eq-lp2}) cannot be satisfied by finite measures, and hence, regardless of the values of the one-stage costs, the linear program proposed in \cite{HoL94} for solving countable-space MDPs is infeasible for the MDP in this example.
\qed
\end{example}

The preceding discussion suggests that the approach we took to study the optimal average cost functions can supplement the existing methods and provide an additional tool for studying average-cost MDPs. 
We now give several more examples to demonstrate the usage of our results.
  
The following example illustrates an application of Prop.~\ref{prp-ac-const-mc} and Markov chain theory to control systems that involve additive random disturbances.
Specifically, we consider a basic MDP model for a single-product inventory system where, at time $n$, the state $x_n$ is the stock level, the action $a_n$ corresponds to the amount of the product ordered, and the disturbance, denoted $\xi_n$, corresponds to the demand/consumption of the product. We consider the case where negative stock levels are also allowed; the nonnegative case can be treated in a similar way.
The one-stage cost $c(x,a)$ is typically the sum of the ordering cost and the holding or shortage cost, minus the sales revenue; its precise definition does not matter here. 

\begin{example} \rm \label{ex-const-3}
Let $\X = \R$ and $\A = [0, +\infty)$. Let $x_{n+1} = x_n + a_n - \xi_n$, where $\xi_n$, $n \geq 0$, are independent, identically distributed (i.i.d.) nonnegative random variables with common probability distribution $F$ and finite mean $m_F > 0$. Assume that $m_F \in A(x)$ for all $x \in \X$, and that $c(\cdot)$ is bounded above. Let $g^* = g_1^*$ or $g_3^*$.

Let $\leb$ denote the Lebesgue measure on $\R$. If $F$ is ``spread-out'' (i.e., some convolution power of $F$ is nonsingular w.r.t.\ $\leb$), then under the stationary policy that always takes the action $a=m_F$, $\{x_n\}$ is a random walk on $\R$ whose increment distribution has zero mean and is ``spread-out.'' Such a random walk is Harris recurrent with an invariant measure $\leb$ (cf.\ \cite[Example 3.4(c)]{Num84} and \cite[pp.\ 247-248]{MeT09}). So either $g^* \equiv - \infty$ or, by Prop.~\ref{prp-ac-const-mc}, for some finite constant $\ell$, $g^* = \ell$ $\leb$-a.e.\ and $g^* \leq \ell$ on $\X$.

If $F$ is not ``spread-out'' (so that, in particular, $F$ is singular w.r.t.\ $\leb$) but for some $\epsilon > 0$, we have $A(x) \supset (m_F - \epsilon, m_F + \epsilon)$ for all $x \in \X$, then we consider the stationary policy $\mu$ that draws actions $a_n$ uniformly from $(m_F - \epsilon, m_F + \epsilon)$. The increment distribution of the random walk $\{x_n\}$ induced by $\mu$ has zero mean and is nonsingular w.r.t.\ $\leb$. So, as in the previous case, the same conclusion for $g^*$ holds.

Finally, we can further infer that $g^*$ is constant everywhere, if $g^*$ is real-valued and the following mild condition on the feasible action sets holds in addition: for each $x \in \X$, the set $D_x : = \{ y \in \R \mid y < x, \, A(y) \supset (x - y) + A(x) \}$ has $\leb(D_x) > 0$. 
Since $g^* = \ell$ $\leb$-a.e.\ and $g^* \leq \ell$, this condition implies that $\sup_{y \in D_x} g^*(y) = \ell$ for all $x \in \X$. 
But $g^*(y) \leq g^*(x)$ if $y \in D_x$ (since for any policy $\pi = (\mu_0, \mu_1, \ldots)$, the policy $\pi' = (\mu_0', \mu_1, \ldots)$ with $\mu_0'(da_0 \nmid y)$ being a translation of $\mu_0(da_0 \nmid x)$ by $(x-y)$ has the property that $\Pr_x^\pi (x_1 \in \cdot) = \Pr_y^{\pi'} (x_1 \in \cdot)$ and hence $J^{(i)}(\pi,x) = J^{(i)}(\pi', y)$, $i \in \{ 1,3\}$). Therefore $g^* \equiv \ell$ in this case.
\qed
\end{example}

In the next two examples, we illustrate applications of Theorem~\ref{thm-ac-const} in the case where there is a special state reachable from any initial state, and in the case where there is a special ``reset action.'' We will focus on the choices of the measure $\lambda$ in condition (i) of Theorem~\ref{thm-ac-const} and the verification of that condition.
For these two examples, we take $g^* = g^*_1$ or $g^*_3$, and to keep the discussion focused, 
\emph{we assume that $g^* \not\equiv - \infty$ and that $c(\cdot)$ is bounded above or satisfies (\ref{eq-exuc1}), so that Assumptions~\ref{cond-ac-n1}-\ref{cond-ac-n2} are both satisfied and $g^*$ is bounded above.}

\begin{example}[the case of a special state] \rm \label{ex-const-4}
Consider an MDP in which there is a special state $x^o$ such that for each $x \in \X$, there is some policy $\pi_x \in \Pi$ with $\Pr^{\pi_x}_x (x_n = x^o \ \text{for some} \ n \geq 0) = 1$. 
Then $g^* \leq g^*(x^o)$ by Theorem~\ref{thm-ac-const}. Moreover, as we show in the proof of this theorem (see Lemma~\ref{lem-reachability-cond}),  there is a policy $\pi^o \in \Pi$ such that $\Pr^{\pi^o}_x (x_n = x^o \ \text{for some} \ n \geq 0) = 1$ for all $x \in \X$.

For each $\pi \in \Pi$, define a probability measure $\lambda_\pi$ on $\U(\X)$ by
$$ \lambda_\pi (B) : = \textstyle{\sum_{n =0}^\infty} 2^{-n-1} \, \Pr_{x^o}^\pi ( x_n \in B), \qquad  B \in \U(\X),$$  
which is a ``discounted occupation measure'' for the policy $\pi$ and the initial state $x^o$. 
The measure $\lambda_\pi$ satisfies condition (i) of Theorem~\ref{thm-ac-const}.
To see this, let $B \in \B(\X)$ with $\lambda_{\pi}(B) > 0$. 
Then there is some $m \geq 0$ such that $\Pr_{x^o}^\pi(x_m \in B) > 0$. 
Let $\pi_B$ be the policy that executes the policy $\pi^o$ until the system visits the state $x^o$, and then, starting from the state $x^o$, applies the policy $\pi$ for $m$ stages, followed by switching back to $\pi^o$ and repeating this procedure. (The precise expression of $\pi_B$ in terms of the stochastic kernels is similar to that given in the proof of Theorem~\ref{thm-ac-const}(a) in Section~\ref{sec-5.1}.)  
Clearly, $\Pr^{\pi_B}_x( \tau_B < \infty) = 1$ for all $x \in \X$. So condition (i) of Theorem~\ref{thm-ac-const} holds for $\lambda = \lambda_\pi$ and $\hat \X = \X$.

Let $\Lambda_{x^o} : = \{\lambda_\pi \mid \pi \in \Pi\}$.
Since by definition $\lambda_\pi(\{x^o\}) = 1/2 > 0$, 
we have, by Theorem~\ref{thm-ac-const}, that 
$g^* = g^*(x^o)$ $\lambda$-a.e., for all occupation measures $\lambda \in \Lambda_{x^o}$.
\qed
\end{example}

\begin{example}[the case of a special action and its generalizations] \rm \label{ex-const-5}
Consider an MDP in which there exist a special action $a^o$, an associated probability measure $p^o \in \P(\X)$, and a Borel subset $\X^o$ of $\X$ such that 
\begin{enumerate}[leftmargin=0.65cm,labelwidth=!]
\item[(i)] for any $x \in \X^o$, $a^o \in A(x)$ and $q(dy \nmid x, a^o) = p^o$; and
\item[(ii)] for each $x \in \X$, there is some policy $\pi_x \in \Pi$ with $\Pr_x^{\pi_x} (\tau_{\X^o} < \infty) = 1$.
\end{enumerate}
Similarly to the previous example, by the proof of Lemma~\ref{lem-reachability-cond} given in Section~\ref{sec-5.1}, property (ii) above implies that there is a policy $\pi^o \in \Pi$ such that $\Pr^{\pi^o}_x (\tau_{\X^o} < \infty) = 1$ for all $x \in \X$.

For each $\pi \in \Pi$, define an occupation measure $\lambda_\pi$ on $\U(\X)$ by
$$ \lambda_\pi (B) : = \textstyle{\sum_{n =0}^\infty} 2^{-n-1} \, \Pr_{p^o}^\pi ( x_n \in B), \qquad  B \in \U(\X).$$  
Then condition (i) of Theorem~\ref{thm-ac-const} holds for $\lambda = \lambda_\pi$ and $\hat \X = \X$. The reasoning for this is similar to that in the previous example: For each Borel set $B$ with $\lambda_\pi(B) > 0$, there is some integer $m \geq 0$ with $\Pr_{p^o}^\pi(x_m \in B) > 0$. 
We let the desired policy $\pi_B$ be the policy that executes $\pi^o$ until the system visits the set $\X^o$, applies the special action $a^o$ at the state $x_{\tau_{\X^o}}$, and then applies the policy $\pi$ for $m$ stages, followed by switching back to $\pi^o$ and repeating this procedure.

No two occupation measures from the set $\Lambda_{p^o} : = \{\lambda_\pi \mid \pi \in \Pi\}$ 
are mutually singular, because $\lambda_\pi(\cdot) \geq \tfrac{1}{2} p^o(\cdot)$ for all $\pi \in \Pi$.  Then, by Theorem~\ref{thm-ac-const} and Remark~\ref{rmk-1}, there is some finite constant $\ell$ such that $g^* \leq \ell$ and for all occupation measures $\lambda \in \Lambda_{p^o}$, $g^* = \ell$ $\lambda$-a.e. 

The above analysis carries over to more general cases where property (ii) still holds but a special action is replaced by a stationary policy instead. First, clearly the preceding conclusion on $g^*$ remains true, if instead of a special action $a^o$, we have a universally measurable mapping $f^o : \X^o \to \A$ such that $f^o(x) \in A(x)$ and $q(dy \nmid x, f^o(x)) = p^o$ for all $x \in \X^o$.
More generally, we can replace $a^o$ with a stationary policy $\mu^o \in \Pi_s$ such that 
\begin{equation} \label{eq-spc-act-gen1}
 \int_\A q(B \mid x, a) \, \mu^o(da \mid x) = p^o(B), \qquad \forall \, B \in \B(\X), \ x \in \X^o.
\end{equation} 
A further generalization is to have, instead of $a^o$ or the equality (\ref{eq-spc-act-gen1}), a stationary policy $\mu^o \in \Pi_s$ such that for some constant $\beta > 0$,
\begin{equation}  \label{eq-spc-act-gen2}
  \int_\A q(B \mid x, a) \, \mu^o(da \mid x) \geq \beta \, p^o(B), \qquad \forall \, B \in \B(\X), \ x \in \X^o.
\end{equation}  
In all these (increasingly more general) cases, with $\Lambda_{p^o}$ being the same set of occupation measures as defined above,
we have that for some finite constant $\ell \geq g^*$, $g^* = \ell$ $\lambda$-a.e.\ for all $\lambda \in \Lambda_{p^o}$.
\qed
\end{example}

\section{Proofs for Section~\ref{sec-basic-opt}}  \label{sec-proofs-basic}

In this section we prove Theorems~\ref{thm-strat-m}-\ref{thm-ac-basic2}.

\subsection{Review of some Helpful Facts} \label{sec-proof-review}

We discuss first a basic fact concerning the completion of measures. It explains why, instead of the strategic measures on $\U(\Omega)$, we can work with their restrictions to $\B(\Omega)$ in proving the optimality results given in Section~\ref{sec-basic-opt}. It will also be useful later in the proofs for Section~\ref{sec-ae-const}, when we deal with Markov chains on $(\X, \U(\X))$ induced by stationary policies. 
We specialize this fact to probability measures on Borel spaces; it holds more generally for measure spaces (cf.\ \cite[Chap.\ 3.3, Problem 3]{Dud02}). We include a proof for the sake of completeness.

\begin{lem} \label{lem-umset}
Let $X$ be a Borel space and $\lambda$ a $\sigma$-finite measure on $(X, \B(X))$. Then for any $B \in \U(X)$, there exist sets $A, C \in \B(X)$ with $A \subset B \subset C$ and $\lambda(C \setminus A) = 0$.
\end{lem}

\begin{proof} 
Let $\rho$ be a probability measure on $\B(X)$ equivalent to $\lambda$. Since $B$ is universally measurable, there exists a Borel set $E$ with $\rho^*(E \triangle B) = 0$, where $\rho^*$ denotes the outer measure w.r.t.\ $\rho$, and $E \triangle B$ denotes the symmetric difference between $E$ and $B$. By \cite[Thm.\ 3.3.1]{Dud02}, there exists a Borel set $F$ with $F \supset E \triangle B$ and $\rho(F) = \rho^*(E \triangle B) = 0$. Let $A = E \setminus F$ and $C = E \cup F$. Then $A, C \in \B(X)$, $A \subset B \subset C$, and $\rho(C \setminus A) = \rho(F) = 0$. Since $\rho$ is equivalent to $\lambda$, $\lambda(C \setminus A) = 0$.
\end{proof}

\begin{rem} \rm \label{rmk-lem-umset}
Lemma~\ref{lem-umset} also shows that for any $\sigma$-algebra $\cE(X)$ such that $\B(X) \subset \cE(X) \subset \U(X)$, if $\lambda'$ is a $\sigma$-finite measure on $\cE(X)$ and $\lambda$ is its restriction to $\B(X)$, then $\lambda'$ is determined by $\lambda$ and coincides on $\cE(X)$ with $\bar \lambda$, the completion of $\lambda$. To see this, let $B \in \cE(X)$ and recall that by the definition of the completion of a measure, $\bar \lambda(B) = \lambda(D)$ for any set $D \in \B(X)$ such that the symmetric difference $B \triangle D$ has outer measure zero w.r.t.\ $\lambda$ \cite[p.~102]{Dud02}. The sets $A$ and $C$ given in Lemma~\ref{lem-umset} are among such sets, since we have $B \triangle A \subset C \setminus A, B \triangle C \subset C \setminus A$, and $\lambda(C \setminus A) = 0$. So $\bar \lambda(B) = \lambda(A)$. Since $\lambda' = \lambda$ on $\B(X)$, we also have $\lambda'(C \setminus A) = 0$; since $A \subset B \subset C$, we then have $\lambda'(B) = \lambda'(A) = \lambda(A)$. Thus $\lambda'(B) = \bar \lambda(B)$ for all $B \in \cE(X)$.
\qed
\end{rem}

We now recount several results about analytic sets, lower semianalytic functions, and Borel measurable functions/stochastic kernels, which will be used in the subsequent proofs. 

Let $X$ and $Y$ be Borel spaces. The following operations on analytic sets result in analytic sets:
\begin{itemize}[leftmargin=0.7cm,labelwidth=!]
\item[(a)] Countable unions and countable intersections: if $\{B_n\}$ is a sequence of analytic sets in $X$, then $\cup_n B_n$ and $\cap_n B_n$ are analytic \cite[Cor.\ 7.35.2]{bs}.
\item[(b)] Borel image and preimages: if $B \subset X$ and $D \subset Y$ are analytic and $f : X \to Y$ is a Borel measurable function, then $f(B)$ and $f^{-1}(D)$ are analytic \cite[Prop.~7.40]{bs}. 
\end{itemize}
These properties of analytic sets are reflected in the properties of lower semianalytic functions, whose lower level sets are analytic, as we recall.
In particular, regarding operations that preserve lower semianalyticness, we have the following \cite[Lemma 7.30]{bs}:
\begin{itemize}[leftmargin=0.7cm,labelwidth=!]
\item[(c)] If $D$ is an analytic set and $f_n: D \to [-\infty, \infty]$, $n \geq 1$, is a sequence of lower semianalytic functions, then
the functions $\inf_n f_n$, $\sup_n f_n$, $\liminf_{n \to \infty} f_n$, and $\limsup_{n \to \infty} f_n$ are also lower semianalytic. 
\item[(d)] If $D$ is an analytic set and $f, g: D \to [-\infty, \infty]$ are lower semianalytic functions, then $f+g$ is lower semi-analytic. In addition, if $f, g \geq 0$ or if $g$ is Borel measurable and $g \geq 0$, then $f g$ is lower semianalytic. 
\item[(e)] If $g:X \to Y$ is Borel measurable and $f: g(X) \to [-\infty, \infty]$ is lower semianalytic, then the composition $f \circ g$ is lower semianalytic.
\end{itemize}
In the above and in what follows, for arithmetic operations involving extended real numbers, we adopt the convention $\infty - \infty = - \infty + \infty = \infty$ and $0\cdot \pm \infty = \pm \infty \cdot 0  =0$.

Regarding integration, the following results will be useful:
\begin{enumerate}[leftmargin=0.7cm,labelwidth=!]
\item[(f)] If $f: X \times Y \to [-\infty, + \infty]$ is lower semianalytic (resp.\ Borel measurable) and $\kappa(dy \,|\, x)$ is a Borel measurable stochastic kernel on $Y$ given $X$, then $\int f(x,y) \, \kappa(dy \,|\, x)$ is a lower semianalytic (resp.\ Borel measurable) function on $X$ \cite[Prop.\ 7.48 and Prop. 7.29]{bs}.
\item[(g)] If $f : X \to [-\infty, + \infty]$ is lower semianalytic (resp.\ Borel measurable), the function $p \mapsto \int f dp$ is lower semianalytic (resp.\ Borel measurable) on $\P(X)$ \cite[Cor.\ 7.48.1 and Cor.~7.29.1]{bs}.
\end{enumerate}

Besides (g), an important property concerning probability measures in $\P(X)$ is:
\begin{enumerate}[leftmargin=0.7cm,labelwidth=!]
\item[(h)] for any analytic set $D \subset X$ and real number $a \geq 0$, the sets $\{ p \in \P(X) \mid p (D) > a\}$ and $\{ p \in \P(X) \mid p (D) \geq a\}$ are analytic (\cite[Lem.~(25)]{BFO74}; see also \cite[Prop.~7.43, Cor.~7.43.1]{bs}).
\end{enumerate}

Finally, we discuss several important results about analytic sets in product spaces and partial minimization of lower semianalytic functions on such spaces. 
Let $D \subset X \times Y$ be an analytic set.
\begin{enumerate}[leftmargin=0.7cm,labelwidth=!] 
\item[(i)] The projection of $D$ on $X$, $\text{proj}_X(D) : = \{ x \!\mid (x,y) \in D \ \text{for some} \ y \in Y \}$, is analytic \cite[Prop.~7.39]{bs}. (As can be seen, this is implied by (b); in fact, it was used to prove (b).)
\item[(j)] The Jankov-von Neumann measurable selection theorem \cite[Prop.\ 7.49]{bs}: there exists an analytically measurable function $\phi: \text{proj}_X(D) \to Y$ such that the graph of $\phi$ lies in $D$ (i.e., $\phi$ is measurable w.r.t.\ the $\sigma$-algebra generated by the analytic subsets of $X$, and $(x, \phi(x)) \in D$ for all $x \in \text{proj}_X(D)$).
\end{enumerate}
For partial minimization of a lower semianalytic function $f(x,y)$ (that is, minimizing $f$ over $y$ for each $x$), by applying (i) and (j) to the level sets or epigraph of $f$, one obtains the following \cite[Props.~7.47, 7.50]{bs}:
\begin{enumerate}[leftmargin=0.7cm,labelwidth=!] 
\item[(k)] If $f: D \to [-\infty, +\infty]$ is lower semianalytic, then the function $f^*: \text{proj}_X(D) \to [-\infty, +\infty]$ given by
\begin{equation} \label{eq-minf}
   f^*(x) = \inf_{y \in D_x} f(x, y), \quad \text{where} \ D_x = \{ y \in Y \mid (x, y) \in D \}, 
\end{equation}   
is also lower semianalytic. Furthermore, let $E^* : = \{ x \in \text{proj}_X(D) \mid  \argmin_{y \in D_x} f(x,y) \not= \varnothing \}$. Then
for every $\epsilon > 0$, there exists a universally measurable function $\phi: \text{proj}_X(D) \to Y$
such that $\phi(x) \in D_x$ for all $x \in \text{proj}_X(D)$ and 
\begin{equation} \label{eq-um-minimizer-1}
     f(x, \phi(x)) = f^*(x), \qquad \forall \,  x \in E^*,
\end{equation}
\begin{equation}   \label{eq-um-minimizer-2}
 f(x, \phi(x)) \leq  \begin{cases}
        f^*(x) + \epsilon & \text{if} \ f^*(x) > - \infty, \\
        -1/\epsilon & \text{if} \ f^*(x) = - \infty,
        \end{cases} \qquad \ \ \  \forall \, x \in \text{proj}_X(D) \setminus E^*.
\end{equation}
\end{enumerate}

We will use (a)-(b) and (f)-(h) frequently in the proof of Theorem~\ref{thm-strat-m}. The latter theorem and (k) are the key proof arguments for Theorems~\ref{thm-ac-basic}-\ref{thm-ac-basic2}.

\subsection{Proof of Theorem~\ref{thm-strat-m}}

Recall that $\S_\star \in \{ \S, \S_m, \S_s\}$ and $\S_\star^0 = \{ p \in \S_\star \mid p_0(p) = \delta_x, \, x \in \X \}$, where $p_0(p)$ stands for the marginal distribution of $x_0$ w.r.t.\ $p$. 

\begin{lem} \label{lem-strat-m0}
If $\S_\star$ is analytic, then so is $\S_\star^0$.
\end{lem}

\begin{proof}
Since the mapping $x \mapsto \delta_x$ is a homeomorphism from $\X$ into $\P(\X)$ \cite[Cor.\ 7.21.1]{bs}, $\P_0: = \{\delta_x \mid x \in \X\}$ is a Borel subset of $\P(\X)$ by \cite[Cor.\ 3.3]{Par67}. 
Define a mapping $\psi : \P(\Omega) \to \P(\X)$ that maps each $p \in \P(\Omega)$ to its marginal distribution of $x_0$. Then $\psi$ is Borel measurable by \cite[Prop.\ 7.26 and Cor.\ 7.29.1]{bs}; consequently, $\psi^{-1} (\P_0)$ is a Borel subset of $\P(\Omega)$.
Since $\S^0_\star = \S_\star \cap \psi^{-1} (\P_0)$, it follows that $\S_\star^0$ is analytic if $\S_\star$ is analytic.
\end{proof}

We now treat the three cases of $\S_\star$ separately and prove that it is analytic in each case (see Props.~\ref{prp-S}, \ref{prp-Sm}, and \ref{prp-Ss}). Together with Lemma~\ref{lem-strat-m0}, this will establish Theorem~\ref{thm-strat-m}.

\subsubsection{The Set $\S$}

Let $n \geq 0$. Denote $h'_n : = (x_0, a_0, \ldots, x_{n}, a_{n})$ and denote its space by $H'_n$; thus $H'_n = (\X \times \A)^{n+1}$.
Recall that $h_n : = (x_0, a_0, x_1, a_1, \ldots, x_n)$ and $H_n  =  (\X \times \A)^n \times \X$ is the space of $h_n$.
With respect to $p \in \P(\Omega)$, the probability of an event $E$ is denoted by $p \{ E\}$.

Consider any $p \in \S$. Then $p$ is induced by some policy, so from the control constraint in the MDP we have
\begin{equation} \label{cond-strm1a}
p \big\{ (x_n, a_n) \in \Gamma \big\} = 1, \qquad \forall \, n \geq 0.
\end{equation}
From the state transition dynamics in the MDP we also have
\begin{equation} \label{cond-strm1b}
    \int_\Omega f_{n,i}(h'_{n}, x_{n+1}) \, p(d \omega) = \int_\Omega \int_\X  f_{n,i}(h'_{n}, y)  \, q(dy \mid x_{n}, a_{n}) \, p(d\omega), \qquad \forall \, i \geq 1, \ n \geq 0,
\end{equation}
where for each $n \geq 0$, $\{f_{n,i}\}_{i \geq 1}$ are the indicator functions of a countable family of Borel subsets of $H_{n+1}$ that form a measure determining class---that is, for any $\rho_1, \rho_2 \in \P(H_{n+1})$, $\rho_1 = \rho_2$ if and only if 
$\int f_{n,i} \, d \rho_1 =  \int f_{n,i} \, d \rho_2$ for all $i \geq 1$. (Such a countable family exists since $\B(H_{n+1})$ is countably generated.)

\begin{lem} \label{lem-strm1}
The set $\S = \big\{ p \in \P(\Omega) \mid p \ \text{satisfies (\ref{cond-strm1a}) and (\ref{cond-strm1b})} \big\}$.
\end{lem}

\begin{proof}
As just discussed, every $p \in \S$ satisfies (\ref{cond-strm1a})-(\ref{cond-strm1b}). 
Consider now any $p \in \P(\Omega) $ that satisfies these constraints.
By a repeated application of \cite[Cor.\ 7.27.2]{bs} to decompose the marginals of $p$ on $H'_n, H_{n+1}$, $n \geq 0$, and taking into account (\ref{cond-strm1b}), we can represent $p$ as the composition of its marginal $p_0(dx_0)$ on $H_0$ with a sequence of Borel measurable stochastic kernels: 
$$p_0(dx_0), \ \mu_0(da_0 \mid x_0), \ q(d x_1 \mid x_0, a_0), \  \ldots, \ \mu_{n}(da_{n} \mid h_{n}),  \ q(d x_{n+1} \mid x_{n}, a_n), \ \ldots.$$
(In other words, $p$ coincides with the unique probability measure on $\B(\Omega)$ determined by the above sequence.)
Define $E_n : = \{ h_n \in H_n \mid \mu_n( A(x_{n}) \,|\, h_n) < 1\big\}$, $n \geq 0$.  
Since the stochastic kernels $\mu_n$ are Borel measurable and the sets $A(x)$, $x \in \X$, are analytic, the sets $E_n$ are universally measurable \cite[Prop.\ 7.46]{bs}. 
Since $p$ satisfies (\ref{cond-strm1a}), we must have 
\begin{equation} \label{eq-lem-strm1-prf1}
    p \{ h_n \in E_n \} = 0, \qquad \forall \,  n \geq 0.
\end{equation}   
Now, for some fixed $\mu^o \in \Pi_s$ and for each $n \geq 0$, let 
$$ \tilde \mu_n(\cdot \mid h_n) : = \begin{cases} 
  \mu_n(\cdot \mid h_n), & \text{if} \  h_n \not \in E_n; \\
  \mu^o (\cdot \,|\, x_n), & \text{if} \   h_n \in E_n.
  \end{cases}
$$ 
Let $\pi: = (\tilde \mu_0, \tilde \mu_1, \ldots)$. Then for all $n \geq 0$, $\tilde \mu_n(da_n \,|\, h_n)$ is a universally measurable stochastic kernel that satisfies $\tilde \mu_n(A(x_n) \,|\, h_n) = 1$ for all $h_n \in H_n$, so $\pi$ is a valid policy in $\Pi$.
By induction on $n$ and using (\ref{eq-lem-strm1-prf1}), it is straightforward to verify that 
$p \big\{ h'_n \in B^{(n)} \big\} = \Pr^\pi_{p_0} \big\{ h'_n \in B^{(n)} \big\}$
for every measurable rectangle $B^{(n)}$ of the form $B^{(n)} = B_0 \times \cdots \times B_{n}$, $B_i \in \B(\X \times \A)$, $0 \leq i \leq n$, $n \geq 0$. 
This implies that $p$ coincides with the restriction of $\Pr^\pi_{p_0}$ to $\B(\Omega)$ and hence belongs to $\S$.
\end{proof}

\begin{prop} \label{prp-S}
The set $\S$ is analytic.
\end{prop}

\begin{proof}
For each $n \geq 0$, let $E_n : = \{p \in \P(\Omega) \mid \text{$p$ satisfies (\ref{cond-strm1a}) for the given $n$}\}$. For each $n \geq 0$ and $i \geq 1$, let $F_{n,i} : = \{ p \in \P(\Omega) \mid  \text{$p$ satisfies (\ref{cond-strm1b}) for the function $f_{n,i}$}\}$. 
By Lemma~\ref{lem-strm1}, $\S = \P_1 \cap \P_2$, for $\P_1 : = \cap_{n \geq 0} E_n$, $\P_2 : = \cap_{n \geq 0, i \geq 1} F_{n,i}$.
Let us show that every $E_n$ is analytic and every $F_{n,i}$ is Borel. 
This will imply that $\S$ is analytic.

For each $n \geq 0$, condition (\ref{cond-strm1a}) is the same as that $p(D) = 1$ for the set $D : = (\X \times \A)^{n} \times \Gamma \times (\X \times \A)^\infty$. Since $\Gamma$ is analytic, $D$ is an analytic subset of $\Omega$ \cite[Prop.~7.38]{bs}. Then, by \cite[Prop.~7.43]{bs}, 
$E_n = \{ p \in \P(\Omega) \mid p(D) = 1\}$ is analytic. 

For each $n \geq 0$ and $i \geq 1$, since $f_{n,i}$ is Borel measurable and $q(dy \,|\, x,a)$ is a Borel measurable stochastic kernel, $\int_\X  f_{n,i}(h'_{n}, y)  \, q(dy \nmid x_{n}, a_{n})$ is a Borel measurable function of $h'_n$ \cite[Prop.~7.29]{bs}. Then, by \cite[Cor.~7.29.1]{bs}, the integral on the right-hand side (r.h.s.)~of (\ref{cond-strm1b}) is a Borel measurable function of $p$, 
and so is the integral on the left-hand side (l.h.s.)~of (\ref{cond-strm1b}) (since $f_{n,i}$ is Borel measurable). This implies that $F_{n,i}$ is a Borel subset of $\P(\Omega)$. The proof is now complete.
\end{proof}

\subsubsection{The Set $\S_m$}

To prove that $\S_m$ is analytic, we will show that it is the image of an analytic set under a Borel measurable mapping (from one Borel space into another). This mapping will be constructed based on the observation that for a Markov policy $\pi$, $\Pr^\pi_{p_0}$ can be determined from its marginal distributions of $(x_n, a_n)$, $n \geq 0$.

To this end, consider the process $\{(x_n, a_n)\}_{n \geq 0}$ induced by some $\pi \in \Pi_m$ and $p_0 \in \P(\X)$. 
Let $\gamma_n$ be the marginal distribution of $(x_n, a_n)$ restricted to $\B(\X \times \A)$. 
Then
$$z: = (p_0, \gamma_0, \gamma_1, \ldots) \in  \P(\X) \times \big( \P(\X \times \A) \big)^\infty =: \Z,$$
and we shall refer to $z$ as the sequence of marginal distributions induced by $\pi$ and $p_0$.
Let $\Delta \subset \Z$ be the set of all such sequences induced by Markov policies and initial distributions. 
It is shown in \cite[Def.\ 9.4 and Prop.\ 9.2]{bs} that $\Delta = \{ z \in \Z \mid z \ \text{satisfies (\ref{cond-Delta1})-(\ref{cond-Delta2}) given below}\}$:
\begin{equation} \label{cond-Delta1}
   \gamma_n(\Gamma) = 1, \qquad \forall \, n \geq 0;
\end{equation}
\begin{equation} \label{cond-Delta2}
 \gamma_0(B \times \A) = p_0(B), \quad   \gamma_n(B \times \A) = \int_{\X \times \A}  q(B \mid x, a)  \, \gamma_{n-1} \big(d(x,a)\big), \quad \forall \, B \in \B(\X), \, n \geq 1. 
\end{equation}

\begin{lem}[{\cite[Lem.\ 1]{ShrB79}; see also \cite[Lem.\ 9.1]{bs}}] 
The set $\Delta$ is analytic. 
\end{lem}

We now construct a Borel measurable mapping that maps $\Delta$ onto $\S_m$.
First, we represent the identity mapping $\gamma \mapsto \gamma$ on $\P(\X \times \A)$ as a Borel measurable stochastic kernel:
Define $\mrho(\cdot \,|\, \cdot) : \B(\X \times \A) \times \P(\X \times \A) \to [0,1]$ by
\begin{equation} \label{def-mrho}
\mrho(B \,|\, \gamma):  = \gamma(B), \qquad B \in \B(\X \times \A), \, \gamma \in \P(\X \times \A).
\end{equation}
Then $\mrho$ is a Borel measurable stochastic kernel on $\X \times \A$ given $\P(\X \times \A)$ (cf.\ \cite[Def.\ 7.12]{bs}), and
by \cite[Cor.\ 7.27.1]{bs}, it can be decomposed as
\begin{equation} \label{eq-dec-meas}
 \mrho(d(x,a) \mid \gamma) = \mrho_2(da \mid x ; \gamma) \, \mrho_1(dx \mid \gamma), 
\end{equation} 
where $\mrho_1$ is a Borel measurable stochastic kernel on $\X$ given $\P(\X \times \A)$; for a fixed $\gamma$, $\mrho_1(dx \mid \gamma)$ is simply the marginal of $\gamma$ on $\X$; and $\mrho_2$ is a Borel measurable stochastic kernel on $\A$ given $\X \times \P(\X \times \A)$. 
We will use these kernels, instead of a direct decomposition of $\gamma$, to construct the mappings we need in this and subsequent analyses, because their explicit dependence on $\gamma$ and their measurability in $\gamma$ make it easier to study the measurability property of the resulting mappings.

Define a mapping $\zeta_m : \Z \to \P(\Omega)$ that maps
each $z = (p_0, \gamma_0, \gamma_1, \ldots) \in  \Z$ to the unique probability measure $p$ on $\B(\Omega)$ that satisfies the following: 
for all $n \geq 0$ and all sets of the form $B^{(n)} : = B_0 \times B_1 \cdots \times B_{n}$ with $B_i \in \B(\X \times \A)$ for $0 \leq i \leq n$,
\begin{align}
 p\{ h'_{n} \in B^{(n)} \} = \int \cdots \int \prod_{i=0}^{n} \ind_{B_i}(x_{i}, a_{i}) \, & \mrho_2(da_{n} \mid x_{n}; \gamma_{n}) \, q(dx_{n} \mid x_{n-1}, a_{n-1}) \, \cdots \notag \\
  \,  \cdots \, \, & \mrho_2(da_1 \mid x_1; \gamma_1) \, q(dx_1 \mid x_0, a_0) \, \mrho(d(x_0, a_0) \mid \gamma_0). \label{def-zetam}
\end{align}  

%\smallskip
\begin{lem} \label{lem-strm2a}
The mapping $\zeta_m : \Z \to \P(\Omega)$ defined through (\ref{def-zetam}) is Borel measurable. 
\end{lem}

\begin{proof}
The $\sigma$-algebra $\B (\Omega)$ is generated by measurable rectangles, the finite disjoint unions of which form an algebra. 
From this and \cite[Prop.~7.26]{bs}, it follows that, for $\zeta_m$ to be a Borel measurable mapping from $\Z$ into $\P(\Omega)$, it suffices that the r.h.s.\ of its defining equation (\ref{def-zetam}) is a (real-valued) Borel measurable function of $z = (p_0, \gamma_0, \gamma_1, \ldots)$, for each set $B^{(n)}$ involved in the definition.
Now, for each $B^{(n)}$, the iterated integral in (\ref{def-zetam}) involves Borel sets $B_i$ and Borel measurable stochastic kernels $q, \mrho$, and $\mrho_2$, so it is a Borel measurable function of $(\gamma_0, \gamma_1, \ldots, \gamma_n)$ by a repeated application of \cite[Prop.~7.29]{bs}. The desired conclusion then follows.
\end{proof}

\begin{lem} \label{lem-strm2b}
The set $\S_m = \zeta_m(\Delta)$. In fact, restricted to $\Delta$, $\zeta_m$ is one-to-one and has a Borel measurable inverse, so $\S_m$ and $\Delta$ are Borel isomorphic.
\end{lem}

\begin{proof}
If $p \in \S_m$ and $z = (p_0, \gamma_0, \gamma_1, \ldots) \in \Delta$ are both induced by a Markov policy $\pi =(\mu_0, \mu_1, \ldots) \in \Pi_m$ and initial distribution $p_0 \in \P(\X)$,  let us prove $p = \zeta_m(z)$. For all $n \geq 0$, since $\gamma_n$ corresponds to the marginal distribution of $(x_n, a_n)$ under $\pi$,  it satisfies that 
$$\gamma_n(B \times D) = \int_{B} \mu_n(D \mid x) \, \mrho_1(dx \mid \gamma_n), \qquad \forall \,  B \in \B(\X), \  D \in \B(\A).$$
By comparing this relation with (\ref{eq-dec-meas}) for $\gamma = \gamma_n$,  there must exist some set $E_n \in \B(\X)$ such that   
\begin{equation} \label{eq-prf-lem-strm2b-1}
  \gamma_n(E_n \times \A) = 0 \quad \text{and} \quad \mu_n(da \,|\, x) = \mrho_2(da \,|\, x; \gamma_n) \ \ \ \forall \,  x \not\in E_n.
\end{equation}  
Then, by induction on $n$ and using (\ref{eq-prf-lem-strm2b-1}) and (\ref{cond-Delta2}), it can be verified 
that the iterated integral in the r.h.s.\ of (\ref{def-zetam}) gives the same value if for every $n \geq 1$, $\mrho_2(\cdot \,|\, \cdot \, ; \gamma_n)$ in this integral is replaced by $\mu_n(\cdot \,|\, \cdot)$. Since $\zeta_m(z)$ is defined by (\ref{def-zetam}),
this implies that $\zeta_m(z)$ coincides with the restriction of $\Pr^\pi_{p_0}$ on $\B(\Omega)$, which is $p$. Thus $p = \zeta_m(z)$, and we have proved $\S_m \subset \zeta_m(\Delta)$. But each point in $\Delta$ can be induced by some Markov policy and initial distribution. Therefore, $\S_m = \zeta_m(\Delta)$.

For any $z = (p_0, \gamma_0, \gamma_1, \ldots) \in \Delta$, by (\ref{cond-Delta2}) and the definition (\ref{def-zetam}) for $\zeta_m$, the marginal distribution of $(x_n, a_n)$ w.r.t.\ $\zeta_m(z)$ is $\gamma_n$. So, if $z, z' \in \Delta$ and $z \not = z'$, then $\zeta_m(z) \not= \zeta_m(z')$.
This shows that $\zeta_{m, \Delta}$, the restriction of $\zeta_m$ to $\Delta$, has an inverse. 
Now, for $p \in \S_m$, $\zeta_{m,\Delta}^{-1}(p)$ is simply given by the sequence $(p_0, \gamma_0, \gamma_1, \ldots)$ where $p_0$ is the marginal distribution of $x_0$ and $\gamma_n$ the marginal distribution of $(x_n, a_n)$ w.r.t.\ $p$. The mapping that maps $p \in \P(\Omega)$ to such a sequence of its marginals is Borel measurable in $p$ by \cite[Prop.\ 7.26 and Cor.\ 7.29.1]{bs}. This proves the lemma.
\end{proof}

\begin{prop} \label{prp-Sm}
The set $\S_m$ is analytic.
\end{prop}

\begin{proof}
Since $\Delta$ is analytic \cite[Lem.~9.1]{bs} and Borel images of analytic sets are also analytic sets \cite[Prop.~7.40]{bs}, Lemmas~\ref{lem-strm2a} and~\ref{lem-strm2b} together imply that $\S_m$ is analytic. 
\end{proof}

\subsubsection{The Set $\S_s$}

Similarly to the preceding case of $\S_m$, to prove $\S_s$ is analytic, we will show that it is the image of an analytic set under a Borel measurable mapping.

Let $G$ be the set of all $\ugam: = (\gamma_0, \gamma_1, \ldots) \in \big(\P(\X \times \A)\big)^\infty$ such that for all $n \geq 0$,
\begin{equation} \label{def-setG}
  \gamma_n (d(x,a)) = \mrho_2(da \mid x; \tilde \gamma) \, \mrho_1(dx \mid \gamma_n), \qquad \text{where} \ \ \tilde \gamma : = \sum_{k=0}^\infty 2^{-k-1} \gamma_{k}.
\end{equation}  
Define a mapping $\zeta_s : \Z \to \P(\Omega)$ that maps
each $z = (p_0, \gamma_0, \gamma_1, \ldots) \in  \Z$ to the unique probability measure $p$ on $\B(\Omega)$ that satisfies the following: with $\tilde \gamma \in \P(\X \times \A)$ be as in (\ref{def-setG}), for all $n \geq 0$ and all sets of the form $B^{(n)} : = B_0 \times B_1 \cdots \times B_{n}$ with $B_i \in \B(\X \times \A)$ for $0 \leq i \leq n$,
\begin{align}
 p\{ h'_{n} \in B^{(n)} \} = \int \cdots \int \prod_{i=0}^{n} \ind_{B_i}(x_{i}, a_{i}) \, & \mrho_2(da_{n} \mid x_{n}; \tilde \gamma) \, q(dx_{n} \mid x_{n-1}, a_{n-1}) \, \cdots \notag \\
  \,  \cdots \, \, & \mrho_2(da_1 \mid x_1; \tilde \gamma) \, q(dx_1 \mid x_0, a_0) \, \mrho_2(da_0 \mid x_0; \tilde \gamma) \, \mrho_1(dx_0 \mid \gamma_0). \label{def-zetas}
\end{align}

\begin{lem} \label{lem-strm3a}
The set $G$ defined by (\ref{def-setG}) is a Borel subset of $\big(\P(\X \times \A)\big)^\infty$. The mapping $\zeta_s : \Z \to \P(\Omega)$ defined through (\ref{def-zetas}) is Borel measurable. 
\end{lem}

\begin{proof}
First, we prove that $\psi: \ugam \mapsto \sum_{k=0}^\infty 2^{-k-1} \gamma_k$ is a Borel measurable mapping from the space $\big( \P(\X \times \A) \big)^\infty$ into $\P(\X \times \A)$.
By \cite[Prop.~7.26]{bs}, it suffices to prove that for each $B \in \B(\X \times \A)$, 
$\psi(\ugam) (B) = \sum_{k=0}^\infty 2^{-k-1} \gamma_k(B)$ 
is a real-valued Borel measurable function of $\ugam$. Now, for each $k \geq 0$, $\gamma_k(B)$ is a Borel measurable function of $\gamma_k$ \cite[Cor.~7.29.1]{bs}, so $\sum_{k=0}^n 2^{-k-1} \gamma_k(B)$ 
is a Borel measurable function of $(\gamma_0, \gamma_1, \ldots, \gamma_n)$ for all $n \geq 0$. Since $\psi(\ugam) (B) = \lim_{n \to \infty} \sum_{k=0}^n 2^{-k-1} \gamma_k(B)$, it is a Borel measurable function of $\ugam$ as desired.

Consider now the set $G$. Let $\{f_i\}_{i \geq 1}$ be the indicator functions of a countable family of Borel subsets of $\X \times \A$ that form a measure determining class. By (\ref{def-setG}), $G = \cap_{n \geq 0} B_n$ where $B_n$ consists of those $\ugam$ that satisfy
\begin{equation} \label{eq-prf-setG1}
   \int _{\X \times \A} f_i \, d\gamma_n = \int_\X \int_\A f_i(x,a) \, \mrho_2\big(da \mid x; \psi(\ugam) \big)\, \mrho_1(dx \mid \gamma_n), \qquad \forall \, i \geq 1.
\end{equation}
For each $i \geq 1$, the l.h.s.\ of (\ref{eq-prf-setG1}) is a Borel measurable function of $\gamma_n$ \cite[Cor.~7.29.1]{bs}, and the r.h.s.\ of (\ref{eq-prf-setG1}) can be written as $\phi\big(\gamma_n, \psi(\ugam)\big)$ for the function 
$$\phi(\gamma, \gamma') : = \int_\X \int_\A f_i(x,a) \, \mrho_2\big(da \mid x; \gamma' \big)\, \mrho_1(dx \mid \gamma), \qquad (\gamma, \gamma') \in \big(\P(\X \times \A)\big)^2.$$
Since the function $f_i$ and the stochastic kernels $\mrho_1, \mrho_2$ are all Borel measurable, $\phi$ is Borel measurable by \cite[Prop.~7.29]{bs}. Combining this with the Borel measurability of $\psi$ proved earlier, it follows that $\phi\big(\gamma_n, \psi(\ugam)\big)$ is a Borel measurable function of $\ugam$. Thus, for each $i \geq 1$, both sides of (\ref{eq-prf-setG1}) are Borel measurable functions of $\ugam$. Consequently, $B_n$ is the intersection of countably many Borel sets and is therefore Borel measurable. Then $G= \cap_{n \geq 0} B_n$ is also Borel measurable. 

Similarly to the proof of Lemma~\ref{lem-strm2a}, by \cite[Prop.~7.26]{bs}, for the mapping $\zeta_s$ to be Borel measurable, it suffices that the iterated integral in its defining equation (\ref{def-zetas}) is a (real-valued) Borel measurable function of $(p_0, \ugam)$, for each set $B^{(n)}$ involved in the definition.
Now for each $B^{(n)}$, similarly to the preceding proof, if we treat $\tilde \gamma$ in (\ref{def-zetas}) as a free variable, then, since the sets $B_i$ and the stochastic kernels involved are all Borel measurable, the iterated integral in (\ref{def-zetas}) is a Borel measurable function of $(\gamma_0, \tilde \gamma)$ by a repeated application of \cite[Prop.~7.29]{bs}.
With $\tilde \gamma = \psi(\ugam)$ and $\psi$ being Borel measurable as proved earlier, it then follows that for each $B^{(n)}$, the iterated integral in (\ref{def-zetas}) is a Borel measurable function of $\ugam$.
This proves that $\zeta_s$ is Borel measurable.
\end{proof}

\begin{lem} \label{lem-strm3b}
The set $\S_s = \zeta_s(\Delta_s)$, where $\Delta_s : = \Delta \cap (\P(\X) \times G)$. 
Moreover, restricted to $\Delta_s$, $\zeta_s$ is one-to-one and has a Borel measurable inverse, so $\S_s$ and $\Delta_s$ are Borel isomorphic.
\end{lem}

\begin{proof}
First, suppose that $p \in \S_s$ and $z = (p_0, \gamma_0, \gamma_1, \ldots) \in \Delta$ are both induced by a stationary policy $\mu \in \Pi_s$  and initial distribution $p_0 \in \P(\X)$.  We show that $\ugam : = (\gamma_0, \gamma_1, \ldots) \in G$ and $p = \zeta_s(z)$. 

Let $\tilde \gamma = \sum_{n = 0}^\infty 2^{-n-1} \gamma_n$. 
Since $\gamma_n$ corresponds to the marginal distribution of $(x_n, a_n)$ under $\mu$,  it satisfies that 
\begin{equation} \label{eq-prf-Ss-1}
 \gamma_n(B \times D) = \int_{B} \mu(D \mid x)  \, \mrho_1(dx \mid \gamma_n), \qquad \forall \, B \in \B(\X), \, D \in \B(\A), \, n \geq 0.
\end{equation} 
By combining these equalities for all $n$, we have
$$  \tilde \gamma (B \times D) = \sum_{n=0}^\infty 2^{-n-1} \gamma_n(B \times D) = \int_{B} \mu(D \mid x) \, \mrho_1(dx \mid \tilde \gamma), \qquad \forall \, B \in \B(\X), \, D \in \B(\A).$$
Comparing this with (\ref{eq-dec-meas}) for $\gamma = \tilde \gamma$, it follows that there exists a set $E \in \B(\X)$ such that
\begin{equation} \label{eq-prf-Ss-2} 
  \tilde \gamma(E \times \A) = 0 \quad \text{and} \quad \mu(da \mid x) = \mrho_2(da \mid x; \, \tilde \gamma) \ \ \   \forall  \, x \not\in E.
\end{equation}   

Equation (\ref{eq-prf-Ss-2}) has two implications. First, since any set of $\tilde \gamma$-measure zero has measure zero w.r.t.\ every $\gamma_n$, (\ref{eq-prf-Ss-2}) together with (\ref{eq-prf-Ss-1}) implies that $\gamma_n$ satisfies (\ref{def-setG}) for every $n \geq 0$. Therefore, $\ugam \in G$ and $(p_0, \ugam) \in \Delta \cap (\P(\X) \times G)$ as desired. 

Secondly, using (\ref{eq-prf-Ss-2}) together with (\ref{def-setG}) and (\ref{cond-Delta2}), it can be verified by induction on $n$ that the iterated integral in the r.h.s.\ of the defining equation (\ref{def-zetas}) for $\zeta_s$ gives the same value if we replace $\mrho_2(\cdot \,|\, \cdot, \tilde \gamma)$ in this integral by $\mu(\cdot \,|\, \cdot)$. This implies that $\zeta_s(p_0, \ugam)$ coincides with the restriction of $\Pr^\mu_{p_0}$ to $\B(\Omega)$, which is $p$, so $p = \zeta_s(p_0, \ugam)$. We have thus proved that 
$\S_s \subset \zeta_s \big(\Delta \cap \big(\P(\X) \times G)\big)$. 

To prove the reverse inclusion, consider any $(p_0, \ugam) \in \Delta$ with $\ugam = (\gamma_0, \gamma_1, \ldots) \in G$. 
Let $\tilde \gamma := \sum_{n = 0}^\infty 2^{-n-1} \gamma_n$. 
Since $(p_0, \ugam) \in \Delta$, by summing up weighted versions of (\ref{cond-Delta1}) over $n$, we have $\tilde \gamma (\Gamma) = 1$. 
Since
$\tilde \gamma(d(x,a)) = \mrho_2(da \,|\, x; \tilde \gamma) \, \mrho_1(dx \,|\, \tilde \gamma)$, this implies the existence of a set $E \in \B(\X)$ such that
\begin{equation} \label{eq-prf-Ss-3} 
  \tilde \gamma(E \times \A) = 0 \quad \text{and} \quad  \mrho_2(A(x) \mid x; \, \tilde \gamma)  = 1 \ \ \   \forall  \, x \not\in E.
\end{equation} 
For some fixed $\mu^o \in \Pi_s$, define
$$ \mu(\cdot \mid x) : = \begin{cases} 
  \mrho_2(\cdot \mid x; \tilde \gamma), & \text{if} \  x \not \in E; \\
  \mu^o (\cdot \,|\, x), & \text{if} \   x \in E.
  \end{cases}
$$ 
Then $\mu$ is a universally measurable stochastic kernel that satisfies $\mu(A(x) \,|\, x) = 1$ for all $x \in \X$ and hence $\mu \in \Pi_s$, and moreover, (\ref{eq-prf-Ss-2}) holds. Applying the same argument given above for the second implication of (\ref{eq-prf-Ss-2}), we have that $\zeta_s(p_0, \ugam)$ coincides with the restriction of $\Pr^{\mu}_{p_0}$ to $\B(\Omega)$, so $\zeta_s(p_0, \ugam) \in \S_s$. This proves that $\zeta_s \big(\Delta \cap (\P(\X) \times G)\big) \subset \S_s$. 
Hence $\S_s = \zeta_s \big(\Delta \cap (\P(\X) \times G)\big)$. 

Finally, notice that if $z = (p_0, \gamma_0, \gamma_1, \ldots) \in \Delta_s$, then by (\ref{cond-Delta2}), (\ref{def-setG}), and the definition (\ref{def-zetas}) for $\zeta_s$, the marginal distribution of $(x_n, a_n)$ w.r.t.\ $\zeta_s(z)$ is $\gamma_n$. 
The second assertion of the lemma then follows from the same argument used in the proof of the second assertion of Lemma~\ref{lem-strm2b} for $\S_m$.
\end{proof}

Since $\Delta$ is analytic \cite[Lem.~9.1]{bs}, by Lemmas~\ref{lem-strm3a} and~\ref{lem-strm3b}, $\S_s$ is the image of an analytic set under a Borel measurable mapping. So we obtain the desired conclusion by \cite[Prop.~7.40]{bs}:

\begin{prop} \label{prp-Ss}
The set $\S_s$ is analytic.
\end{prop}

\subsection{Proofs of Theorems~\ref{thm-ac-basic}-\ref{thm-ac-basic2}}

The main proof arguments for Theorems~\ref{thm-ac-basic}-\ref{thm-ac-basic2}
can be summarized as follows. First, express the average cost problems (under various criteria involved in the two theorems) equivalently as minimization problems on the set of probability measures induced by the policies in $\Pi$, $\Pi_m$, or $\Pi_s$. 
Show that these minimization problems correspond to partial minimization of lower semianalytic functions on analytic sets.
Then, apply a measurable selection theorem for such partial minimization problems (cf.\ (k) in Section~\ref{sec-proof-review}) to obtain measurable ($\epsilon$-)optimal solution mappings from $\X$ into certain sets of induced probability measures. From these mappings, construct policies that have the desired optimality properties.

We now give the proofs. Let $\S_\star \in \{\S, \S_m, \S_s\}$. 
Define sets $\tilde \S_\star \subset \P(\X) \times \P(\Omega)$, $\tilde \S^0_\star \in \X \times \P(\Omega)$ by
$$\tilde \S_\star : = \{ (p_0(p), p) \mid p \in \S_\star \}, \qquad \tilde \S^0_\star : = \{ (x, p) \mid \delta_x = p_0(p), x \in \X, p \in \S^0_\star \},$$
where $p_0(p)$ stands for the marginal distribution of $x_0$ w.r.t.\ $p$.

\begin{lem} \label{lem-ext-set}
The sets $\tilde \S_\star$ and $\tilde \S_\star^0$ are analytic.
\end{lem}

\begin{proof}
The set $\tilde \S_\star = \psi_1(\S_\star)$, where $\psi_1 : p \mapsto (p_0(p), p)$ is a Borel measurable mapping from $\P(\Omega)$ into $\P(\X) \times \P(\Omega)$. The set $\tilde \S_\star^0 = \psi_2 \big( \psi_1 (\S_\star^0) \big)$, where $\psi_2 : (\delta_x, p) \mapsto (x, p)$ is a homeomorphism from $\{ \delta _x \mid x \in \X\} \times \P(\Omega)$ onto $\X \times \P(\Omega)$.
Thus $\tilde \S_\star, \tilde \S^0_\star$ are Borel images of the analytic sets $\S_\star, \S^0_\star$ (Theorem~\ref{thm-strat-m}), respectively, so by \cite[Prop.~7.40]{bs}, they are analytic.
\end{proof}

The next proposition relates a universally measurable selection on $\tilde \S^0_\star$ to a universally measurable policy with certain structures. 
Its proof is long and will be given after we first use it to prove the two theorems. 
We remark that although not needed in this paper, there is an analogue of this proposition for $\tilde \S_\star$ by essentially the same proof arguments.

Denote by $\tilde \S^0_\star(x)$ the $x$-section of $\tilde \S^0_\star$; that is,  
$$\tilde \S^0_\star(x) : = \{ p \in \P(\Omega) \mid (x, p) \in \tilde \S^0_\star \} = \{ p \in \S^0_\star \mid p_0(p) = \delta_x \}.$$

\begin{prop} \label{prp-ummap-pol}
Suppose that $\zeta : \X \to \P(\Omega)$ is a universally measurable mapping such that  $\zeta(x) \in \tilde \S^0_\star(x)$ for all $x \in \X$. 
Then there exists a universally measurable policy $\pi \in \Pi$ such that: 
\begin{enumerate}[leftmargin=0.7cm,labelwidth=!]
\item[\rm (i)] for all $x \in \X$, $p_x = \zeta(x)$ where $p_x$ is the restriction of $\Pr^\pi_x$ to $\B(\Omega)$; 
\item[\rm (ii)] in the case $\tilde \S^0_\star = \tilde \S^0_m$ (resp.\ $\tilde \S^0_\star = \tilde \S^0_s$), $\pi$ is semi-Markov (resp.\ semi-stationary).
\end{enumerate}
\end{prop}

We now rewrite the average cost problems involved in Theorems~\ref{thm-ac-basic}-\ref{thm-ac-basic2} as partial minimization problems on $\tilde \S^0_\star$.
Consider any average cost criterion $J^{(i)}$, $1 \leq i \leq 4$. We have
\begin{equation} \label{eq-partial-min}
   g^*(x) = g^m(x) = \inf_{p \in \tilde \S_m^0(x)} f(p), \qquad  g^s(x) =  \inf_{p \in \tilde \S_s^0(x)} f(p), \quad   x \in \X,
\end{equation}   
where we dropped the subscript $i$ for ``$g$'' to simplify notation, and the function $f: \S^0 \to [- \infty, + \infty]$ is defined according to the criterion $J^{(i)}$ under consideration as follows:
\begin{align}  
\text{for $J^{(1)}$}: & \qquad f(p) : = \limsup_{n \to \infty} n^{-1} \int_\Omega \textstyle{ \sum_{k=0}^{n-1} c(x_k, a_k)} \, p(d \omega),  \label{eq-def-f1a} \\
\text{for $J^{(3)}$}: & \qquad    f(p) : = \limsup_{n \to \infty} \sup_{j \geq 0} n^{-1} \int_\Omega \textstyle{\sum_{k=0}^{n-1} c(x_{k+j}, a_{k+j})} \, p(d \omega), \qquad \qquad  \label{eq-def-f1b}
\end{align}    
and the cases of $J^{(2)}$ and $J^{(4)}$ are similar. The equality $g^*(x) = g^m(x)$ in (\ref{eq-partial-min})
comes from the well-known fact that for any initial distribution $p_0$ and history-dependent policy $\pi \in \Pi$, there exists a Markov policy $\pi_m \in \Pi_m$ under which the marginal distributions of $(x_n, a_n)$, $n \geq 0$, coincide with those under $\pi$ (cf.\ the proof of \cite[Prop.~1]{ShrB79}) and hence, the average costs of $\pi$ and $\pi_m$ are equal at $p_0$, w.r.t.\ any $J^{(i)}$. 

Likewise, for any criterion $\tJ^{(i)}, 1 \leq i \leq 4$, letting $\tg$ stand for $\tg_i$, we have
\begin{equation} \label{eq-partial-min2}
 \tg^*(x) = \inf_{p \in \tilde \S^0(x)} f(p),   \qquad \tg^\diamond(x) = \inf_{p \in \tilde \S_\diamond^0(x)} f(p), \qquad x \in \X,  \ \ \diamond \in \{m, s\},
\end{equation} 
where 
the function $f: \S^0 \to [- \infty, + \infty]$ is defined according to the criterion $\tJ^{(i)}$:
\begin{align}  
\text{for $\tJ^{(1)}$}:  & \qquad  f(p) : =  \int_\Omega \left( \limsup_{n \to \infty} n^{-1} \textstyle{ \sum_{k=0}^{n-1} c(x_k, a_k)} \right) \, p(d \omega),  \label{eq-def-f2a} \\
\text{for $\tJ^{(3)}$}: & \qquad  f(p) : =  \int_\Omega \left( \limsup_{n \to \infty}  \sup_{j \geq 0} n^{-1} \textstyle{ \sum_{k=0}^{n-1} c(x_{k+j}, a_{k+j})} \right) \, p(d \omega),   \qquad \qquad \label{eq-def-f2b} 
\end{align}   
and the other two cases are similar.

For the classes of MDPs considered in Theorems~\ref{thm-ac-basic}-\ref{thm-ac-basic2} (\mdp, \mdn, \ptmdp, \ptmdn), 
the functions $f$ given above are all well defined and do not involve $+\infty - \infty$, since $f(p)$ is just another way to express $J(\pi,x)$ or $\tJ(\pi, x)$. In general, we can adopt the convention $+\infty - \infty = +\infty$ to have 
the following lemma hold without restrictions on the MDP model.

\begin{lem} \label{lem-f-lsa}
For $J^{(i)}, \tJ^{(i)}, 1 \leq i \leq 4$, the corresponding functions $f$ are lower semianalytic.
\end{lem} 

\begin{proof}
This lemma follows from Theorem~\ref{thm-strat-m} and the preservation of lower semianalyticness under various operations (cf.\ Section~\ref{sec-proof-review}).
We give the proof details for $f$ defined in (\ref{eq-def-f1b}) and (\ref{eq-def-f2b}); the other cases are similar. 
First, for technical convenience, we extend the lower semianalytic one-stage cost function $c(\cdot)$ from $\Gamma$ to $\X \times \A$ by defining $c(x,a) : = +\infty$ on $\Gamma^c$. This extension is also lower semianalytic. 
Now for each $n \geq 1, j \geq 0$, the function $\phi_{n,j}(\omega) : = \sum_{k=0}^{n-1} c(x_{k+j}, a_{k+j})$ is lower semianalytic on $\Omega$ by \cite[Lem.\ 7.30(4)]{bs}, so the function $\psi_{n,j}(p) : = \int \phi_{n,j} \, dp$ is lower semianalytic on $\P(\Omega)$ by \cite[Cor.\ 7.48.1]{bs}. 
For $f$ given by (\ref{eq-def-f1b}), we have $f(p) = \limsup_{n \to \infty} \sup_{j \geq 0} n^{-1} \psi_{n,j}(p)$ and the domain of $f$, $\S^0$, is an analytic subset of $\P(\Omega)$ (Theorem~\ref{thm-strat-m}). It then follows from \cite[Lem.\ 7.30(2)]{bs} that $f$ is lower semianalytic.
In the case of (\ref{eq-def-f2b}), the domain of $f$ is the same analytic set $\S^0$, and we have $f(p) = \int \phi \, dp$, for the function $\phi(\omega) : = \limsup_{n \to \infty} \sup_{j \geq 0} n^{-1} \phi_{n,j}(\omega)$, which is lower semianalytic on $\Omega$ by \cite[Lem.\ 7.30(2)]{bs}. Then $\int \phi \, dp$ is a lower semianalytic function of $p$ by \cite[Cor.\ 7.48.1]{bs}, and it follows that $f$ is lower semianalytic. 
\end{proof}

We are now ready to prove Theorems~\ref{thm-ac-basic}-\ref{thm-ac-basic2}. 

\begin{proof}[Proof of Theorems~\ref{thm-ac-basic}-\ref{thm-ac-basic2}]
The arguments are the same for the two theorems. 
We have shown that the sets $\tilde \S^0, \tilde \S_m^0$, and $\tilde \S_s^0$ are analytic (Lemma~\ref{lem-ext-set}),
and that for each case of the average cost criterion, the corresponding function $f(p)$ on $\S^0$ is lower semianalytic (Lemma~\ref{lem-f-lsa}), which obviously implies that viewed as a function of $(x,p)$, 
$\hat f (x,p) : = f(p)$ is lower semianalytic on $\X \times \S^0  \supset \tilde \S^0$.
Thus every minimization problem appearing in (\ref{eq-partial-min}) and (\ref{eq-partial-min2}) is of the form of partial minimization of a lower semianalytic function on an analytic set.  By \cite[Prop.\ 7.47]{bs} the resulting function from such partial minimization is lower semianalytic (cf.\ (k) in Section~\ref{sec-proof-review}). 
So the functions $g^\star, \tg^\star$, $\star \in \{\ast, m, s\}$ are all lower semianalytic, as asserted in parts (i) of the two theorems.

Furthermore, by a measurable selection theorem \cite[Prop.\ 7.50]{bs}, for any $\epsilon > 0$, there exist measurable $\epsilon$-optimal solution mappings for these partial minimization problems (cf.\ (k) in Section~\ref{sec-proof-review}). More precisely, consider, for instance, the two partial minimization problems in (\ref{eq-partial-min}). 
By \cite[Prop.\ 7.50]{bs}, for every $\epsilon > 0$, there is a universally measurable mapping $\zeta: \X \to \P(\Omega)$ such that
for all $x \in \X$,
\begin{equation} \label{eq-thm-basic-prf1}
\zeta(x) \in \tilde \S_\diamond^0(x) \quad \text{and} \quad f\big(\zeta(x)\big) \leq \begin{cases} 
    g^\diamond(x) + \epsilon & \text{if} \ g^\diamond(x) > - \infty, \\
    - \epsilon^{-1} & \text{if} \ g^\diamond(x) = - \infty,
    \end{cases}
\end{equation}
and 
\begin{equation} \label{eq-thm-basic-prf2}
     f(\zeta(x)) = g^\diamond(x)  \quad \text{on} \ \big\{ x \in \X \mid \exists \, p \in \tilde \S_\diamond^0(x) \ s.t.\ f(p) = g^\diamond(x) \big\},
\end{equation}
where $\diamond \in \{m, s\}$. 
Let $\pi$ be the policy given by Prop.~\ref{prp-ummap-pol} for the above mapping $\zeta$. Then $\pi$ is semi-Markov (resp.\ semi-stationary) in the case $\diamond = m$ (resp.\ $\diamond = s$) by Prop.~\ref{prp-ummap-pol}(ii), and by Prop.~\ref{prp-ummap-pol}(i), 
$f(\zeta(x)) = J^{(i)}(\pi, x)$ for all $x \in \X$. This together with (\ref{eq-thm-basic-prf1}) and (\ref{eq-thm-basic-prf2}) proves that $\pi$ fulfills the requirements in Theorem~\ref{thm-ac-basic}. (In the case of the second statement in Theorem~\ref{thm-ac-basic}(ii), we also use the fact that w.r.t.\ the criteria $J^{(i)}$, if there exists an optimal policy for a state $x$, then there exists a Markov optimal policy for that state).

Theorem~\ref{thm-ac-basic2}(ii)-(iii) follows from applying the same argument given above to the three partial minimization problems in (\ref{eq-partial-min2}). Specifically, for Theorem~\ref{thm-ac-basic2}(ii), we apply the above argument with $\tilde S^0$ in place of $\tilde \S^0_\diamond$ and with $\tg^*$ in place of $g^\diamond$; and for Theorem~\ref{thm-ac-basic2}(iii), we apply the above argument with $\tg^\diamond$ in place of $g^\diamond$.
\end{proof}

In the rest of this subsection, we prove Prop.~\ref{prp-ummap-pol}, which we used in the above proof.

\begin{proof}[Proof of Prop.~\ref{prp-ummap-pol}]
(i) Consider first the case $\tilde \S^0_\star = \tilde \S^0$. Let us construct a randomized history-dependent policy $\pi \in \Pi$ that satisfies Prop.~\ref{prp-ummap-pol}(i). The argument is similar to the one used to prove Lemma~\ref{lem-strm1}, except that here the probability measures induced by $\pi$, when restricted to $\B(\Omega)$, must agree with $\zeta(x)$ for all initial distributions $\delta_x, x \in \X$, instead of a single initial distribution.

Since $\zeta$ is universally measurable, $\zeta(x)(d\omega)$ is a universally measurable stochastic kernel on $\Omega$ given $\X$ \cite[Def.\ 7.12]{bs}. By a repeated application of \cite[Prop.\ 7.27]{bs} to decompose the marginals of $\zeta(x)(d\omega)$ on $H_n$ and $H'_n$, $n \geq 0$, and by also taking into account that $\zeta(x) \in \tilde \S^0(x) \subset \S^0$ so that (\ref{cond-strm1b}) holds for all $\zeta(x), x \in \X$ (cf.\ Lemma~\ref{lem-strm1}),
we can represent $\zeta(x)(d \omega)$ as the composition of the marginal distribution of $x_0$ with a sequence of stochastic kernels:
$$ p_0(dx_0 \mid x) = \delta_x(dx_0), \ \mu_0(da_0 \mid x_0; \, x), \ q(dx_1 \mid x_0, a_0), \ \ldots, \mu_n(da_n \mid h_n; \, x), \ q(dx_{n+1} \mid x_n, a_n), \ \ldots.$$
In the above, except for those related to state transitions, all the stochastic kernels have parametric dependences on $x$. 
Moreover, $p_0(dx_0 \mid x) = \delta_x(dx_0)$ is Borel measurable in $x$, and
for every $B \in \B(\A)$, $\mu_n(B \,|\, h_n; \, x)$ is $\big(\B(H_n) \otimes \U(\X)\big)$-measurable in $(h_n, x)$ by \cite[Prop.\ 7.27]{bs}, which implies that $\mu_n$ is a universally measurable stochastic kernel on $\A$ given $H_n \times \X$ \cite[Lem.~7.28]{bs}.

We now modify $\mu_n$ to satisfy the control constraint of the MDP. 
For $n \geq 0$ and $x \in \X$, let 
$$D_n : = \{ (h_n, x) \in H_n \times \X \mid  \mu_n( A(x_n) \mid h_n ; \, x) < 1 \}, \qquad D_{n,x} : = \{ h_n \mid (h_n, x) \in D_n \};$$
since $\mu_n$ is universally measurable and the sets $A(y)$, $y \in \X$, are analytic, $D_n$ and $D_{n,x}$ are universally measurable \cite[Prop.\ 7.46]{bs}.
Since $\zeta(x) \in \S$, by (\ref{cond-strm1a}), $\zeta(x)\{(x_n, a_n) \in \Gamma\} = 1$ for all $n \geq 0$, so we must have
\begin{equation} \label{eq-ummap-prf1}
\zeta(x)\{ h_n \in D_{n,x}\} = 0,  \qquad \forall \, n \geq 0, \ x \in \X. 
\end{equation}
For $n=0$, since $\zeta(x) \in \tilde \S^0(x)$, (\ref{eq-ummap-prf1}) implies that $\mu_0(A(x) \,|\, x; \, x) = 1$, i.e., $(x,x) \not\in D_0$, for all $x \in \X$.

For some fixed $\mu^o \in \Pi_s$, define
\begin{equation} \label{eq-ummap-prf2}
   \tilde \mu_n(d a_n \mid h_n) : = \begin{cases} 
          \mu_n(d a_n \mid h_n; \, x_0), & \text{if} \ (h_n, x_0) \not\in D_n; \\
           \mu^o(d a_n \mid x_n), & \text{if} \ (h_n, x_0)  \in D_n.
           \end{cases}        
\end{equation}
In this definition, since $\mu_n$ is a universally measurable stochastic kernel and the mapping $\psi: h_n \mapsto (h_n, x_0)$ is Borel measurable, for every $B \in \B(\A)$, $\mu_n(B \nmid h_n; \, x_0)$ is universally measurable in $h_n$ \cite[Prop.~7.44]{bs}; and the two sets, $\{ h_n \in H_n \mid (h_n, x_0) \not\in D_n\}$ and $\{ h_n \in H_n \mid (h_n, x_0) \in D_n\}$, are preimages of the universally measurable sets $D_n^c$ and $D_n$ under the Borel measurable mapping $\psi$ and are therefore universally measurable \cite[Cor.\ 7.44.1]{bs}. 
Consequently, $\tilde \mu_n (da_n \,|\, h_n)$ is a universally measurable stochastic kernel on $\A$ given $H_n$ that satisfies $\tilde \mu_n(A(x_n) \mid h_n) = 1$ for all $h_n \in H_n$.
Then $\pi : = (\tilde \mu_0, \tilde \mu_1, \ldots )$ is a valid (universally measurable) policy in $\Pi$.

To show that $\pi$ satisfies Prop.~\ref{prp-ummap-pol}(i), it suffices to show that for each $x \in \X$, we have
\begin{equation} \label{eq-ummap-prf3}
    \Pr^\pi_x \big\{ h'_n \in B^{(n)} \big\} = \zeta(x)\{ h'_n \in B^{(n)} \}
\end{equation}    
for every set $B^{(n)}$ of the form $B^{(n)} = B_0 \times \cdots \times B_{n}$, $B_i \in \B(\X \times \A)$, $0 \leq i \leq n$, $n \geq 0$.
We prove this by induction on $n$. For $n = 0$, since $(x,x) \not\in D_0$ as discussed earlier, we have
$$     \Pr^\pi_x \big\{ h'_0 \in B_0 \big\}  = \int_\A \ind_{B_0} (x, a_0) \, \tilde \mu_0(da_0 \mid x)  = \int_\A \ind_{B_0} (x, a_0) \, \mu_0(da_0 \mid x; \, x) = \zeta(x) \{ h_0' \in B_0 \}.$$
Now suppose that (\ref{eq-ummap-prf3}) holds for some $n \geq 0$. 
Then the marginal distribution of $\Pr^\pi_x$ on $H'_n$, denoted $p^{(n)}_x$, coincides, on $\B(H'_n)$, with the marginal distribution of $\zeta(x)$ on $H'_n$, denoted $\zeta^{(n)}_x$. 
Write $B^{(n+1)}$ as $B^{(n+1)} = B^{(n)} \times B_{n+1}$, and let $m: = n+ 1$. We have 
\begin{align}
     \Pr^\pi_x \big\{ h'_{m} \in B^{(m)} \big\} & = \int_{B^{(n)}} \int_\X \int_\A \ind_{B_{m}}(x_{m}, a_{m}) \, \tilde \mu_{m}(da_{m} \mid h_{m}) \, q(dx_{m} \mid x_n, a_n) \, p^{(n)}_x( dh'_n) \notag \\
     & = \int_{B^{(n)}} \int_\X \int_\A \ind_{B_m}(x_m, a_m) \, \tilde \mu_m(da_m \mid h_m) \, q(dx_m \mid x_n, a_n) \, \zeta^{(n)}_x( dh'_n) \notag \\ 
   & = \int_{B^{(n)}} \int_\X \left\{ \int_\A \ind_{B_m}(x_m, a_m)  \,  \mu_m(da_m \mid h_m; \, x_0) \right\} \cdot \ind( h_m \not\in D_{m, x_0} ) \notag \\ 
   & \qquad \qquad \quad \cdot q(dx_m \mid x_n, a_n) \, \zeta^{(n)}_x( dh'_n) \notag  \\
    & \ \  \ + \int_{B^{(n)}} \int_\X \left\{ \int_\A \ind_{B_m}(x_m, a_m) \,  \mu^o(da_m \mid x_m) \right\} \cdot \ind( h_m \in D_{m, x_0} )  \notag \\
    & \qquad \qquad \qquad \cdot  q(dx_m \mid x_n, a_n) \, \zeta^{(n)}_x( dh'_n). \notag 
\end{align}
Since $\zeta(x)\{x_0 = x\} = 1$ and $\zeta(x)\{ h_m \in D_{m,x}\} = 0$ by (\ref{eq-ummap-prf1}), we have
$$ \int_{H'_n} \int_\X  \ind( h_m \in D_{m, x_0} )  \, q(dx_m \mid x_n, a_n) \, \zeta^{(n)}_x( dh'_n) = 0,$$
which together with the preceding derivation leads to
\begin{align*}
 \Pr^\pi_x \big\{ h'_m \in B^{(m)} \big\}  & = \int_{B^{(n)}} \int_\X  \int_\A \ind_{B_m}(x_m, a_m)  \,  \mu_m(da_m \mid h_m; \, x) \, q(dx_m \mid x_n, a_n) \, \zeta^{(n)}_x( dh'_n)  \\
 & = \zeta(x)\{ h'_m \in B^{(m)} \}. 
\end{align*} 
This completes the induction and proves that (\ref{eq-ummap-prf3}) holds all $n \geq 0$.  Thus $\pi$ satisfies Prop.~\ref{prp-ummap-pol}(i).

\medskip
\noindent (ii) Consider now the case $\tilde \S^0_\star =  \tilde \S^0_s$. 
For each $x \in \X$, let $\gamma_n(x)$ be the marginal distribution of $(x_n, a_n)$ w.r.t.\ $\zeta(x)$, and 
let  $z(x) : = (\delta_x, \gamma_0(x), \gamma_1(x), \ldots)$. 
Let $\zeta_{s,\Delta_s}$ denote the restriction of $\zeta_s$ to $\Delta_s$, where $\zeta_s$ is defined through (\ref{def-zetas}) and $\Delta_s$ is as defined in Lemma~\ref{lem-strm3b}. 
Since $\zeta(x) \in \S_s^0 \subset \S_s$ by assumption,  
by Lemma~\ref{lem-strm3b} and its proof, $\zeta(x)$ and $z(x)$ are related as follows: $z(x) \in \Delta_s$,
$\zeta(x) = \zeta_s(z(x))$ and $z(x) = \zeta_{s,\Delta_s}^{-1} \big(\zeta(x) \big)$, where $\zeta_{s,\Delta_s}^{-1} : \S_s \to \Delta_s$ is Borel measurable. Then, since $\zeta$ is universally measurable, by \cite[Prop.~7.44]{bs}, $z(x)$ is universally measurable in $x$.

Let $\tilde \gamma(x) : =  \sum_{n=0}^\infty 2^{-n-1} \gamma_n(x)$. Consider a stochastic kernel $\mu$ on $\A$ given $\X^2$ defined by
$$\mu(da \mid y,  x)  : = \mrho_2\big(da \mid y; \, \tilde \gamma(x)\big),$$
where $\mrho_2$ is defined in (\ref{eq-dec-meas}) as we recall. 
Since $\mrho_2$ is a Borel measurable stochastic kernel and the mapping $x \mapsto \tilde \gamma(x)$ is universally measurable (since it is the composition of the universally measurable mapping $x \mapsto z(x)$ with a Borel measurable mapping; cf.\ the start of the proof of Lemma~\ref{lem-strm3a}), 
the stochastic kernel $\mu$ is universally measurable.
Then, in view of the fact that the sets $A(y), y \in \X$, are analytic, the sets
$$D : = \{(y, x) \in \X^2 \mid \mu(A(y) \mid y,  x) < 1 \}, \qquad D_x : = \{ y \in \X \mid (y,x) \in D\}, \quad x \in \X,$$ 
are universally measurable by \cite[Prop.\ 7.46]{bs}.
Since $z(x) \in \Delta_s$, from (\ref{cond-Delta1}) and (\ref{def-setG}), 
we have
\begin{equation} \label{eq-ummap-prf4}
 \gamma_n(x)(D_x \times \A) = 0, \qquad \forall \, n \geq 0, \  x \in \X.
\end{equation} 
For $n=0$, (\ref{eq-ummap-prf4}) is the same as $\delta_x(D_x) = 0$, i.e., $\mu(A(x) \,|\, x, x) = 1$ so that $(x,x) \not\in D$, for all $x \in \X$.

We now modify $\mu$ to define a semi-stationary policy. For some fixed $\mu^o \in \Pi_s$, let
\begin{equation} \label{eq-ummap-prf5}
   \tilde \mu(d a \mid y, x_0) : = \begin{cases} 
          \mu(d a \mid y,  x_0), & \text{if} \ (y, x_0) \not\in D; \\
           \mu^o(d a \mid y), & \text{if} \ (y, x_0)  \in D.
           \end{cases}        
\end{equation}
Then $\tilde \mu(A(y) \,|\, y, x_0) = 1$ for all $(y, x_0) \in \X^2$, and since the stochastic kernels $\mu, \mu^o$ and the sets $D^c, D$ are all universally measurable, $\tilde \mu$ is a universally measurable stochastic kernel. 
Consequently, $\pi : = (\tilde \mu(da_0 \,|\, x_0, x_0), \, \ldots, \, \tilde \mu(da_n \,|\, x_n, x_0), \, \ldots)$ is a semi-stationary (universally measurable) policy.

Finally, we prove that $\pi$ satisfies Prop.~\ref{prp-ummap-pol}(i). As in part (i) of this proof, it suffices to show that (\ref{eq-ummap-prf3}) holds all $n \geq 0$ and all measurable rectangles $B^{(n)}$, and we will prove this by induction on $n$. 
For $n= 0$, since $(x,x) \not\in D$ and $\zeta(x)\{x_0 = x\} = \gamma_0(x)(\{x\} \times \A) = 1$, we have
$$     \Pr^\pi_x \big\{ h'_0 \in B_0 \big\}  = \int_\A \ind_{B_0} (x, a_0) \, \tilde \mu(da_0 \mid x, x)  = \int_\A \ind_{B_0} (x, a_0) \, \mrho_2\big(da_0 \mid x;  \, \tilde \gamma(x) \big) = \zeta(x) \{ h_0' \in B_0 \},$$
where the last equality follows from the relation $\zeta(x) = \zeta_s(z(x))$ and the definition (\ref{def-zetas}) for $\zeta_s$.
Now suppose that (\ref{eq-ummap-prf3}) holds for some $n \geq 0$. Then, as in part (i) of this proof, we have that $p^{(n)}_x$ coincides with $\zeta^{(n)}_x$ 
on $\B(H'_n)$, where $p^{(n)}_x$ and $\zeta^{(n)}_x$ are the marginals of $\Pr^\pi_x$ and $\zeta(x)$ on $H'_n$, respectively.
Writing $B^{(n+1)}$ as $B^{(n+1)} = B^{(n)} \times B_{n+1}$, we have that for $m: = n+1$,
\begin{align}
     \Pr^\pi_x \big\{ h'_m \in B^{(m)} \big\} & = \int_{B^{(n)}} \int_\X \int_\A \ind_{B_m}(x_m, a_m) \, \tilde \mu(da_m \mid x_m, x_0) \, q(dx_m \mid x_n, a_n) \, p^{(n)}_x( dh'_n) \notag \\
     & = \int_{B^{(n)}} \int_\X \int_\A \ind_{B_m}(x_m, a_m) \, \tilde \mu(da_m \mid x_m, x_0) \, q(dx_m \mid x_n, a_n) \, \zeta^{(n)}_x( dh'_n) \notag \\ 
        & = \int_{B^{(n)}} \int_\X \left\{ \int_\A \ind_{B_m}(x_m, a_m)  \,  \mrho_2\big(da_m \mid x_m; \, \tilde \gamma(x_0)\big) \right\} \cdot \ind( x_m \not\in D_{x_0} ) \notag \\ 
   & \qquad \qquad \quad \cdot q(dx_m \mid x_n, a_n) \, \zeta^{(n)}_x( dh'_n) \notag  \\
    & \ \  \ + \int_{B^{(n)}} \int_\X \left\{ \int_\A \ind_{B_m}(x_m, a_m) \,  \mu^o(da_m \mid x_m) \right\} \cdot \ind( x_m \in D_{x_0} )  \notag \\
    & \qquad \qquad \qquad \cdot  q(dx_m \mid x_n, a_n) \, \zeta^{(n)}_x( dh'_n). \notag 
\end{align}
Since $\zeta(x)\{x_0 = x\}=1$ and $\gamma_m(x)(D_{x} \times \A) = 0$  by (\ref{eq-ummap-prf4}), we have
$$ \int_{H'_n} \int_\X  \ind( x_m \in D_{x_0} )  \, q(dx_m \mid x_n, a_n) \, \zeta^{(n)}_x( dh'_n) = \int_{\X \times \A}  \ind( x_m \in D_{x} ) \,\gamma_m(x) (d(x_m, a_m))  = 0$$
(where we also used the relation (\ref{cond-Delta2}) between $\gamma_{n+1}(x)$ and $\gamma_n(x)$ to obtain the first equality). 
Combing this with the preceding derivation, we have
\begin{align*}
 \Pr^\pi_x \big\{ h'_m \in B^{(m)} \big\}  & = \int_{B^{(n)}} \int_\X  \int_\A \ind_{B_m}(x_m, a_m)  \,   \mrho_2\big(da_m \mid x_m; \, \tilde \gamma(x_0)\big)  \, q(dx_m \mid x_n, a_n) \, \zeta^{(n)}_x( dh'_n)  \\
 & = \zeta(x)\{ h'_m \in B^{(m)} \},
\end{align*} 
where the last equality follows from the relation $\zeta(x) = \zeta_s(z(x))$ and the definition (\ref{def-zetas}) of $\zeta_s$.
The induction is now complete; thus (\ref{eq-ummap-prf3}) holds all $n \geq 0$. This proves that the semi-stationary policy $\pi$ satisfies Prop.~\ref{prp-ummap-pol}(i).

\medskip
\noindent (iii) The case $\tilde \S^0_\star =  \tilde \S^0_m$ is similar to the previous case $\tilde \S^0_s$, so we will only outline the proof.  
With $z(x): = (\delta_x, \gamma_0(x), \gamma_1(x), \ldots)$ defined as in part (ii) of this proof, we use the assumption on $\zeta$ and Lemma~\ref{lem-strm2b} to obtain that
$z(x) \in \Delta$, $\zeta(x) = \zeta_m(z(x))$, and $z(x) =  \zeta_{m, \Delta}^{-1}(\zeta(x))$ is universally measurable in $x$. 
(Here $\zeta_m$ is defined in (\ref{def-zetam}) as we recall, and $\zeta_{m, \Delta}$ is the restriction of $\zeta_m$ to $\Delta$.)

For $n \geq 0$, let $\mu_n (da_n \,|\, x_n; \, x) : = \mrho_2\big(da_n \,|\, x_n; \, \gamma_n(x)\big)$, and let 
$$D_n : = \{(x_n, x) \in \X^2 \mid \mu_n(A(x_n) \mid x_n; \,  x) < 1 \}, \qquad D_{n,x} : = \{ y \in \X \mid (y,x) \in D_n\}, \quad x \in \X.$$ 
By the same reasoning given before (\ref{eq-ummap-prf4}), for all $n \geq 0$, $\mu_n$ is a universally measurable stochastic kernel and the sets $D_n, D_{n,x}$ are universally measurable. Since $z(x) \in \Delta$, by (\ref{cond-Delta1}), 
\begin{equation} \label{eq-ummap-prf6}
 \gamma_n(x)(D_{n,x} \times \A) = 0, \qquad \forall \, n \geq 0, \ x \in \X.
\end{equation} 
For some fixed $\mu^o \in \Pi_s$, let 
\begin{equation} 
   \tilde \mu_n(d a_n \mid x_n; \, x_0) : = \begin{cases} 
          \mu_n(d a_n \mid x_n; \,  x_0), & \text{if} \ (x_n, x_0) \not\in D; \\
           \mu^o(d a_n \mid x_n), & \text{if} \ (x_n, x_0)  \in D.
           \end{cases}        \notag
\end{equation}
Then $\pi : = (\tilde \mu_0, \tilde \mu_1, \ldots)$ is a universally measurable, semi-Markov policy.

To show that $\pi$ satisfies Prop.~\ref{prp-ummap-pol}(i), as in the previous cases, we prove (\ref{eq-ummap-prf3}) by induction on $n$. 
In particular, we use (\ref{eq-ummap-prf6}), the relation $\zeta(x) = \zeta_m(z(x))$, and the definition (\ref{def-zetam}) of $\zeta_m$ to complete the induction procedure.  
\end{proof}

\section{Proofs for Section~\ref{sec-ae-const}} \label{sec-5}

In this section we prove the results given in Section~\ref{sec-ae-const}.                        

\subsection{Proofs of Lemma~\ref{lem-ineq} and Theorems~\ref{thm-ac-const}-\ref{thm-ac-const2}} \label{sec-5.1}

Recall that in these theorems, we deal with an MDP in the \mdn class under the average cost criterion $J^{(1)}$ or $J^{(3)}$, and $g^*$ denotes the optimal average cost function $g^*_1$ or $g^*_3$, depending on the criterion under consideration.
We start with two inequalities obtained by applying Fatou's lemma:

\begin{lem} \label{lem-fatou}
Under Assumpion~\ref{cond-ac-n1},
for any $x \in \X$, $a \in A(x)$, and $\pi \in \Pi$,
\begin{align}
\limsup_{n \to \infty} \int_\X n^{-1} J_n(\pi, y) \, q(dy \mid x, a) 
& \leq  \int_\X  \limsup_{n \to \infty} n^{-1} J_n(\pi, y) \, q(dy \mid x, a),  \label{eq-fatou1} \\
\limsup_{n \to \infty} \int_\X  \sup_{j \geq 0} \, n^{-1} J_{n,j}(\pi, y) \, q(dy \mid x, a) 
& \leq  \int_\X  \limsup_{n \to \infty} \sup_{j \geq 0} \, n^{-1} J_{n,j}(\pi, y) \, q(dy \mid x, a). \label{eq-fatou2}
\end{align}   
\end{lem}

\begin{proof}
Under Assumption~\ref{cond-ac-n1}, for all $n \geq 1$, as functions of $y$, $n^{-1} J_{n}(\pi, y)$ and $\sup_{j \geq 0} n^{-1} J_{n,j}(\pi, y)$ are bounded above by the function $M_\pi(y)$, which is integrable w.r.t.\ $q(dy \,|\, x , a)$.
Thus we can apply Fatou's lemma to the l.h.s.\ of (\ref{eq-fatou1}) and (\ref{eq-fatou2}) 
to interchange the order of limit and integral. This yields (\ref{eq-fatou1}) and (\ref{eq-fatou2}).
\end{proof}

Lemma~\ref{lem-ineq} (i.e., the inequality (\ref{eq-ac-ineq}): $g^*(x) \leq \inf_{a \in A(x)} \int_\X g^*(y) \, q(dy \,|\, x, a)$ for all $x \in \X$) follows from the preceding lemma, Theorem~\ref{thm-ac-basic}, and some calculations:

\begin{proof}[Proof of Lemma~\ref{lem-ineq}]
Let $\epsilon > 0$. For the average cost criterion $J^{(i)}$, $i=1$ or $3$, let $\pi^\epsilon$  be a semi-Markov $\epsilon$-optimal policy 
for the corresponding average cost problem; such a policy exists by Theorem~\ref{thm-ac-basic}(ii). 
For each $x \in \X$ and $a \in A(x)$, consider a policy $\pi \in \Pi$ that applies action $a$ at $x$ at the first stage and applies $\pi^\epsilon$ thereafter; i.e., if $\pi^\epsilon = \big(\mu^\epsilon_0(da_0 \,|\, x_0),  \, \mu^\epsilon_1(da_1 \,|\, x_1, x_0), \,  \mu^\epsilon_2(da_2 \,|\, x_2, x_0),  \, \ldots\big)$, then 
$\pi$ can be expressed as 
$$\pi = \big(\, \mu_0(da_0 \mid x_0),  \, \mu^\epsilon_0 (da_1 \mid x_1), \, \mu^\epsilon_1(da_2 \mid x_2, x_1),  \,\mu^\epsilon_1(da_3 \mid x_3, x_1),  \, \ldots \, \big).$$ 
with $\mu_0(da_0 \,|\, x) = \delta_a$. (Such $\mu_0$ exists: for $x_0 \not = x$, let $\mu_0(\cdot \,|\, x_0) = \mu(\cdot \,|\, x_0)$ for a fixed stationary policy $\mu \in \Pi_s$; then $\mu_0$ is a universally measurable stochastic kernel satisfying the control constraint.)

Now consider the case $i = 1$, where the average cost $J^{(1)}(\pi, x) = \limsup_{n \to \infty} n^{-1} J_n(\pi, x)$ by definition. 
By our choice of $\pi$,
$n^{-1} J_n(\pi,x) = n^{-1} c(x,a) + \int_\X n^{-1} J_{n-1}(\pi^\epsilon, y) \, q(dy \nmid x, a)$. Letting $n \to \infty$ and applying Lemma~\ref{lem-fatou} to $ \pi^\epsilon$, we obtain that (regardless of whether $c(x,a) = - \infty$ or not)
\begin{equation} \label{eq-prp-ineq-prf2}
   J^{(1)}(\pi, x) \leq \int_\X  \limsup_{n \to \infty}  n^{-1} J_{n}(\pi^\epsilon, y) \, q(dy \mid x, a) =  \int_\X J^{(1)}(\pi^\epsilon, y) \, q(dy \mid x, a).
\end{equation}
Let $E : = \{ y \in \X \mid g^*_1(y) = - \infty \}$. If $q(E \,|\, x,a) = 0$, then, since $\pi^\epsilon$ is $\epsilon$-optimal, we have $J^{(1)}(\pi^\epsilon, y) \leq g^*_1(y) + \epsilon$ for $q(dy \,|\, x, a)$-almost all $y$, so (\ref{eq-prp-ineq-prf2}) implies 
$g^*_1(x) \leq \int_\X g^*_1(y) \, q(dy \,|\, x,a) + \epsilon$. 
If $q(E \,|\, x,a) \not = 0$, let $t : = \int_\X M_{\pi^\epsilon}(y) \, q(dy \,|\,x, a) < \infty$ (cf.\ Assumption~\ref{cond-ac-n1}). Then, using Assumption~\ref{cond-ac-n1} and the $\epsilon$-optimality of $\pi^\epsilon$, we can bound the r.h.s.\ of (\ref{eq-prp-ineq-prf2}) from above by
$ - \epsilon^{-1} q(E \,|\, x, a) + t$, which, by letting $\epsilon \to 0$, implies $g^*_1(x) = - \infty$. Thus, in either case, $g^*_1(x) \leq \int_\X g^*_1(y) \, q(dy \,|\, x,a) + \epsilon$. Since $\epsilon$ and $a$ are arbitrary, the desired inequality (\ref{eq-ac-ineq}) follows. 

Consider now the case $i=3$, where the average cost $J^{(3)}(\pi, x) = \lim_{n \to \infty} \sup_{j \geq 0} n^{-1} J_{n,j}(\pi, x)$. 
We have 
\begin{equation} \label{eq-prp-ineq-prf3}
  J^{(3)}(\pi, x) \leq \max \left\{ \limsup_{n \to \infty} n^{-1} J_{n,0}(\pi, x), \, \limsup_{n \to \infty} \sup_{j \geq 1} n^{-1} J_{n,j}(\pi, x) \right\}.
\end{equation}   
Since $J_{n,0}(\pi,x) = J_n(\pi,x)$, the same proof argument leading to (\ref{eq-prp-ineq-prf2}) shows that
\begin{equation} \label{eq-prp-ineq-prf4}
  \limsup_{n \to \infty} n^{-1} J_{n,0}(\pi, x) \leq \int_\X  \limsup_{n \to \infty}  n^{-1} J_{n}(\pi^\epsilon, y) \, q(dy \mid x, a).
\end{equation}
For $j \geq 1$, $J_{n,j}(\pi,x) = \int_\X J_{n,j-1}(\pi^\epsilon, y) \, q(dy \,|\, x,a)$, and
\begin{equation}
   \sup_{j \geq 1} \int_\X n^{-1} J_{n,j-1}(\pi^\epsilon, y) \, q(dy \,|\, x,a)  \leq  \int_\X \sup_{j \geq 0} \, n^{-1} J_{n,j}(\pi^\epsilon, y) \, q(dy \,|\, x,a). \notag
\end{equation}   
Letting $n \to \infty$ in this inequality and applying Lemma~\ref{lem-fatou} to $\pi^\epsilon$, we have
\begin{equation}
  \limsup_{n \to \infty} \sup_{j \geq 1} \, n^{-1} J_{n,j}(\pi, x)   \leq
 \int_\X \limsup_{n \to \infty} \sup_{j \geq 0} \, n^{-1} J_{n,j}(\pi^\epsilon, y) \, q(dy \,|\, x,a). \label{eq-prp-ineq-prf5}
\end{equation} 
Combining (\ref{eq-prp-ineq-prf3}), (\ref{eq-prp-ineq-prf4}), and (\ref{eq-prp-ineq-prf5}), we obtain
$$  J^{(3)}(\pi, x) = \lim_{n \to \infty} \sup_{j \geq 0} \, n^{-1} J_{n,j}(\pi, x) \leq \int_\X J^{(3)}(\pi^\epsilon, y) \, q(dy \mid x, a).$$ 
To establish the desired inequality (\ref{eq-ac-ineq}) for $g^*_3$, we can now apply exactly the same proof given immediately after (\ref{eq-prp-ineq-prf2}) for the case $i=1$, with $g^*_3$ in place of $g^*_1$.
\end{proof}

We now proceed to prove Theorem~\ref{thm-ac-const}.
First, let us consider the process $\{(x_n, a_n)\}$ induced by a policy $\pi \in \Pi$ and an initial state $x_0=x \in \X$.
Let $\F_n$ denote the $\sigma$-algebra generated by the state and action variables up to time $n$; i.e., $\F_n$ is generated by the random variable $h'_n(\omega) : = \{ (x_k(\omega), a_k(\omega))\}_{k \leq n}$, which is a measurable mapping from $\Omega$ to $(\X \times \A)^{n+1}$, with both spaces equipped with the universal $\sigma$-algebras.

\begin{lem} \label{lem-subm}
Under Assumptions~\ref{cond-ac-n1} and~\ref{cond-ac-n2}, for any $\pi \in \Pi$ and $x_0 = x \in \X$ such that $g^*(x) \not = - \infty$, $\big\{g^*(x_n), \F_n\big\}_{n \geq 0}$ is a submartingale satisfying $\sup_{n \geq 0} \E^\pi_x \big[ \big| g^*(x_n) \big| \big] <  + \infty$ and, therefore, converges almost surely to an integrable random variable.
\end{lem}

\begin{proof}
Let $[ g^*(\cdot)]_+$ and $[ g^*(\cdot)]_-$ denote the positive and negative parts of $g^*(\cdot)$, respectively. 
Assumption~\ref{cond-ac-n2}(a) implies that $[ g^*(\cdot)]_+\leq M(\cdot)$; then by Assumption~\ref{cond-ac-n2}(b) we have 
\begin{equation} \label{eq-lem-subm-prf}
\textstyle{\sup_{n \geq 0} \E_x^\pi \big\{ [ g^*(x_n) ]_+ \big\} < + \infty.}
\end{equation} 
Lemma~\ref{lem-ineq} implies that under Assumption~\ref{cond-ac-n1}, for all $n \geq 0$,
\begin{equation} \label{eq-lem-subm-prf2}
\E^\pi_x \big[ g^*(x_{n+1}) \mid \F_n \big] \geq g^*(x_n) \qquad \text{and} \qquad \E^\pi_x \big[ g^*(x_{n}) \big] \geq g^*(x),
\end{equation}
where the expectations are well defined in view of (\ref{eq-lem-subm-prf}).
The first relation in (\ref{eq-lem-subm-prf2}) shows that $\big\{g^*(x_n), \F_n\big\}_{n \geq 0}$ is a submartingale. The second relation in (\ref{eq-lem-subm-prf2}), together with (\ref{eq-lem-subm-prf}) and the assumption $g^*(x) \not= - \infty$, implies $\sup_{n \geq 0} \E^\pi_x \big\{ [ g^*(x_n) ]_- \big\} <  + \infty$. 
Hence $\sup_{n \geq 0} \E^\pi_x \big[ \big| g^*(x_n) \big| \big] <  + \infty$. 
Then, by a submartingale convergence theorem \cite[Thm.\ IV-1-2]{Nev75}, $\{g^*(x_n)\}$ converges almost surely to an integrable random variable.
\end{proof}

Recall that for a set $B \subset \X$, the stopping time $\tau_B : = \min \{ n \geq 0 \mid x_n \in B \}$.

\begin{lem} \label{lem-reachability-cond} 
Let $\lambda$ and $\hat \X$ satisfy condition (i) of Theorem~\ref{thm-ac-const}. Then for each Borel set $B \subset \hat \X$ with $\lambda(B) > 0$, there exists a policy $\pi_B \in \Pi$ such that $\Pr_x^{\pi_B} (\tau_B < \infty) = 1$ for all $x \in \hat \X$.
\end{lem}

\begin{proof}
Consider an arbitrary Borel set $B \subset \hat \X$ with $\lambda(B) > 0$.
Define an MDP by modifying the state transition stochastic kernel and the one-stage costs of the original MDP as follows: 
$$
\tilde q(dy \nmid x, a) = \begin{cases} 
    \delta_x & \text{if}  \ x \in B, \, a \in \A; \\
    q(dy \mid x, a) & \text{if} \ x \not\in B, \, a \in \A;
    \end{cases} \qquad \ \ \ 
    \tilde c(x,a) = \begin{cases}
       - 1  & \text{if} \  x \in B, \, a \in A(x); \\
       0 & \text{if} \ x \not\in B, \, a \in A(x).
       \end{cases}
$$
Here the stochastic kernel $\tilde q (dy \nmid x,a)$ is Borel measurable, and the one-stage cost function $\tilde c(\cdot)$ is lower semianalytic and bounded on $\Gamma$. 
Our assumption on $\lambda$ and $\hat \X$ implies that in this modified MDP, w.r.t.\ the average cost criterion $\tJ^{(1)}$, there is an optimal policy for each $x \in \hat \X$, with the optimal average cost being $\tg^*_1(x) = -1$. 
Now by the proof of Theorem~\ref{thm-ac-const2}(ii), there is a universally measurable policy $\tilde \pi \in \Pi$
 that attains the optimal average cost $\tg^*_1(x)$ at every state $x$ for which there exists an optimal policy. 
Thus, in the modified MDP, $\tJ^{(1)}(\tilde \pi, x) = \tg^*_1(x) = -1$ on $\hat \X$. This implies that in the original MDP, $\Pr^{\tilde \pi}_x ( \tau_B < \infty) = 1$ for all $x \in \hat \X$, so we can let $\tilde \pi$ be the desired $\pi_B$.
\end{proof}

\begin{proof}[Proof of Theorem~\ref{thm-ac-const}]
For part (a), we use proof by contradiction. Since $g^*$ is lower semianalytic by Theorem~\ref{thm-ac-basic}(i), $g^*$ is universally measurable. 
Suppose that $g^*$ is not constant $\lambda$-a.e. 
Then there exist constants $\ell_1 < \ell_2$ such that the sets 
$$B'_1 : =\{x \in \X \mid g^*(x) < \ell_1\}, \qquad B'_2 : =\{x \in \X \mid g^*(x) > \ell_2 \}$$ 
satisfy $\lambda(B'_1) > 0, \lambda(B'_2) > 0$. Since $B'_1, B'_2 \in \U(\X)$, by Lemma~\ref{lem-umset} and the assumption $\lambda({\hat \X}^c) = 0$, there are Borel sets $B_1 \subset B'_1 \cap \hat \X$ and $B_2 \subset B'_2 \cap \hat \X$ such that $\lambda(B_1) = \lambda(B'_1) $ and $\lambda(B_2) = \lambda(B'_2)$. 

Let $\pi_{B_1}, \pi_{B_2} \in \Pi$ be two policies given by Lemma~\ref{lem-reachability-cond}, for the sets $B_1, B_2$, respectively. Consider a policy $\pi$ that executes $\pi_{B_1}$ until the system visits some state in $B_1$, then switches to executing $\pi_{B_2}$ until the system visits $B_2$, and then switches back to $\pi_{B_1}$, and so on. More precisely, let $\tau_0 = 0$ and define, recursively, stopping times $\tau_k$, $k \geq 1$, by 
$$\tau_{k} : = \min \big\{ n \geq \tau_{k-1} \mid x_n \in B_i \big\},$$ 
where $i = 1$ if $k$ is odd, and $i = 2$ if $k$ is even. 
If $\pi_{B_1} = (\mu^1_0, \mu^1_1, \ldots)$ and $\pi_{B_2} = (\mu^2_0, \mu^2_1, \ldots)$, then the policy $\pi = (\mu_0, \mu_1, \ldots)$ is given by: for each $n \geq 0$ and $(x_0, a_0, \ldots, a_{n-1}, x_n) \in (\X \times \A)^n \times \X$,
\begin{align*}
& \mu_n( d a_n \mid x_0, a_0, \ldots, a_{n-1}, x_n ) \\
 &  = \begin{cases}
                \mu^1_{n-j} (d a_n \mid x_j, a_j, \ldots, a_{n-1}, x_n )  \quad &   \text{if} \ \, j = \tau_{k} \leq n < \tau_{k+1} \ \text{for some even} \ k \geq 0;   \\
                 \mu^2_{n-j} (d a_n \mid x_j, a_j, \ldots, a_{n-1}, x_n )   & \text{if} \ \, j = \tau_{k} \leq n < \tau_{k+1} \ \text{for some odd} \ k \geq 1.   
                 \end{cases} 
\end{align*}                 
This expression also shows that $\pi$ is universally measurable and therefore a valid policy in $\Pi$.
 
Consider now the process $\{(x_n, a_n)\}$ induced by $\pi$ and an initial state $x_0 = x \in \hat \X$ with $g^*(x) \not= - \infty$ (such a state exists by condition (ii) of the theorem). By Lemma~\ref{lem-reachability-cond} and the construction of $\pi$, both $B_1$ and $B_2$ are visited infinitely often, almost surely. But by Lemma~\ref{lem-subm} $g^*(x_n)$ converges almost surely, which is impossible in view of the definitions of $B_1, B_2$.  This contradiction proves that $g^*$ must be constant $\lambda$-a.e.

Next, we show that $g^* \not= - \infty$ $\lambda$-a.e. If this were false, then, similarly to the preceding proof, we can find a Borel set $B \subset \{ x \in \hat \X \mid g^*(x) = - \infty \big\}$ with $\lambda(B) > 0$ and a corresponding policy $\pi_B$ given by Lemma~\ref{lem-reachability-cond} for the set $B$.
Consider the process $\{(x_n, a_n)\}$ induced by $\pi_B$ and an initial state $x_0 = x \in \hat \X$ with $g^*(x) > - \infty$. 
By Lemma~\ref{lem-subm}, $\{g^*(x_n), \F_n\}$ is a submartingale. By an optional stopping theorem for submartingales \cite[Thm.\ 10.4.1]{Dud02}, 
\begin{equation} \label{eq-thm-ac-prf1}
     \E^{\pi_B}_x \big[ g^*(x_{\tau_B \wedge N}) \big] \geq g^*(x), \qquad \forall \, N \geq 1.
\end{equation}
On the other hand, since $\Pi^{\pi_B}_x (\tau_B < \infty) = 1$, for sufficiently large $N$, $\Pr^{\pi_B}_x (\tau_B \leq N) > 0$ and therefore, 
$$  \E^{\pi_B}_x \big[ g^*(x_{\tau_B \wedge N}) \big]  = (- \infty) \cdot \Pr^{\pi_B}_x (\tau_B \leq N)  + \E^{\pi_B}_x \big[ g^*(x_N) \ind(\tau_B > N) \big] = - \infty.$$
(The above calculation is valid since $\E^{\pi_B}_x \big[ g^*(x_N) \ind(\tau_B > N) \big] \leq \E^{\pi_B}_x [ M(x_N) ] < + \infty$ by Assumption~\ref{cond-ac-n2}.)
This contradicts (\ref{eq-thm-ac-prf1}) since $g^*(x) > - \infty$. Thus we must have $g^* = \ell_\lambda$ $\lambda$-a.e., for some finite constant $\ell_\lambda$.

We now prove part (b) of the theorem. 
Let $\{(x_n, a_n)\}$ be induced by 
an initial state $x_0 = x \in \X$ and a policy $\pi$ that satisfy the assumption in part (b). 
Recall that $D = \{y \in \hat \X \mid g^*(y) = \ell_\lambda \}$.
Define a stopping time $\hat \tau : = \tau_D \wedge \min \{ n \geq 0 \mid g^*(x_n) > \ell_\lambda\} \leq \tau_D$. 

Consider first the case $g^*(x) > - \infty$. 
By Lemma~\ref{lem-subm}, $\{g^*(x_n), \F_n\}$ is a submartingale; by assumption, for some nonnegative function $f \geq g^*$, $\{f(x_n)\}$ are uniformly integrable, so the positive parts of this submartingale, $\{ [ g^*(x_n) ]_+\}$, are uniformly integrable.
Then, by an optional stopping theorem \cite[Cor.\ IV-4-25]{Nev75}, almost surely,
\begin{equation} \label{eq-thm-ac-prf2}
     g^*(x_{\hat \tau}) \leq  \E^{\pi}_x \big[ g^*(x_{\tau_D})  \mid \F_{\hat \tau} \big],
\end{equation}
where $\F_{\hat \tau}$ is the $\sigma$-algebra associated with the stopping time $\hat \tau$. 
Since by assumption $\tau_D < \infty$ a.s., the r.h.s.\ of (\ref{eq-thm-ac-prf2}) equals $\ell_\lambda$ a.s. 
In view of the definition of $\hat \tau$, this implies that $\hat \tau = \tau_D$ a.s., proving the assertion that for all $n \geq 0$,
$g^*(x_{n \wedge \tau_D}) \leq \ell_\lambda$ a.s. For $n=0$, this yields $g^*(x) = g^*(x_0) \leq \ell_\lambda$. 

The case $g^*(x) = - \infty$ is similarly proved: Let $s \vee t : = \max\{s, t\}$ for two extended real numbers $s$ and $t$. Consider the process $Z_n : = g^*(x_n) \vee b$, $n \geq 0$, for some finite negative number $b < \ell_\lambda$. Lemma~\ref{lem-ineq} and Assumption~\ref{cond-ac-n2} imply that $\{Z_n\}$ is a submartingle, and the uniform integrability assumption on $\{f(x_n)\}$ implies that the positive parts $\{ [ Z_n]_+\}$ of this submartingale are uniformly integrable. So, by \cite[Cor.\ IV-4-25]{Nev75}, 
$Z_{\hat \tau} \leq \E^{\pi}_x \big[ Z_{\tau_D}  \mid \F_{\hat \tau} \big]$ a.s.; that is, almost surely,
$$    g^*(x_{\hat \tau}) \vee b \leq  \E^{\pi}_x \big[ g^*(x_{\tau_D}) \vee b \mid \F_{\hat \tau} \big].$$
The same argument given immediately after (\ref{eq-thm-ac-prf2}) then shows that $g^*(x_{n \wedge \tau_D}) \leq \ell_\lambda$ a.s.\ for all $n \geq 0$. 

Finally, consider the last statement in part (b): $g^* \leq \ell_\lambda$ on $\hat \X$ in the special case where $g^*$ is bounded above. 
It follows from the general statement in part (b) that we just proved, by choosing the required policy $\pi$ for a state $x \in \hat \X$ to be the policy $\pi_B$ given by condition (i) of the theorem for a Borel set $B \subset D$ with $\lambda(B) = \lambda(D) > 0$. The existence of such a set $B$ follows from Lemma~\ref{lem-umset}. This completes the proof.
\end{proof}

For the \ptmdn model and the average cost criteria $\tJ^{(i)}, 1 \leq i \leq 4$, Theorem~\ref{thm-ac-const2} is proved by the same arguments given above. In fact, the proof is simpler because for the \ptmdn model, the average cost functions $\tg^\star$, where $\tg^\star \in \{\tg^*_i, \tg^m_i \mid 1 \leq i \leq 4\}$, are bounded from above. 
Thus, provided that $\tg^\star$ satisfies the inequality
\begin{equation} \label{eq-thm2-ac-prf1}
 \tg^\star (x) \leq  \inf_{a \in A(x)} \int_\X \tg^\star_i(y) \, q(dy \mid x, a), \qquad \forall \, x \in \X,
\end{equation}
we have that under any policy $\pi \in \Pi$ and for any $x_0 = x \in \X$ such that $\tg^\star(x) > - \infty$, $\tg^\star(x_n)$, $n \geq 0$, form a submartingale that is bounded above by some constant. The conclusion of Lemma~\ref{lem-subm} then holds for $\tg^\star$, and the same proof of Theorem~\ref{thm-ac-const} carries through with $\tg^\star$ in place of $g^*$. 

Regarding the inequality (\ref{eq-thm2-ac-prf1}), for $\tg^\star = \tg^*_i$, as mentioned in Section~\ref{sec-ae-const-a} (cf.\ (\ref{eq-ac2-eq})), equality actually holds. Working with the average costs of a policy along sample paths, one can prove the equality (\ref{eq-ac2-eq}) by direct calculations similar to those given in the proof of Lemma~\ref{lem-ineq}; we therefore omit the details.
For the case $\tg^\star = \tg^m_i$, let us verify that the inequality (\ref{eq-thm2-ac-prf1}) holds under the assumptions of Theorem~\ref{thm-ac-const2}.

\begin{lem} \label{lem-ineq2} 
Consider the \ptmdn model and any criterion $\tJ^{(i)}, 1 \leq i \leq 4$. If for every $\epsilon > 0$, $\tg^m_i$ can be attained within $\epsilon$ accuracy by a Markov policy, then (\ref{eq-thm2-ac-prf1}) holds for $\tg^\star = \tg^m_i$. 
\end{lem} 

\begin{proof}
For $\epsilon > 0$, let 
$\pi^\epsilon : = \big(\mu^\epsilon_0(da_0 \,|\, x_0),  \, \mu^\epsilon_1(da_1 \,|\, x_1), \,  \mu^\epsilon_1(da_2 \,|\, x_2),  \, \ldots\big)$
be a Markov policy that attains $\tg^m_i$ within $\epsilon$ accuracy. For each $x \in \X$ and $a \in A(x)$, with $\mu_0 \in \Pi_s$ and $\mu_0(da_0 \,|\, x) = \delta_a$, define a policy  
$\pi : = \big(\, \mu_0(da_0 \mid x_0),  \, \mu^\epsilon_0 (da_1 \mid x_1), \, \mu^\epsilon_1(da_2 \mid x_2),  \,\mu^\epsilon_2(da_3 \mid x_3),  \, \ldots \, \big).$
Then $\pi \in \Pi_m$, so $\tJ^{(i)}(\pi, x) \geq \tg^m_i(x)$ by the definition of $\tg^m_i$. 
Since $\tJ^{(i)}(\pi,x) = \int_\X \tJ^{(i)}(\pi^\epsilon, y) \, q(dy \,|\, x, a)$ (which can be verified directly), the desired inequality (\ref{eq-thm2-ac-prf1}) follows, similarly to the proof  of Lemma~\ref{lem-ineq}.
\end{proof}

This establishes Theorem~\ref{thm-ac-const2}, as discussed earlier.

\subsection{Proofs of Lemma~\ref{lem-ac-cond-mc}, Proposition~\ref{prp-ac-const-mc} and Details for Remark~\ref{rmk-mc-related}} \label{sec-5.2}

We start by providing the proof steps needed in order to apply Markov chain theory for state spaces with countably generated $\sigma$-algebras to the case where state spaces are equipped with universal $\sigma$-algebras. This analysis involves some concepts and standard terminology for irreducible Markov chains, which are explained in Appendix~\ref{appsec-mc}.

Let $\mu \in \Pi_s$ and $\tilde \X \in \U(\X)$ be the stationary policy and the absorbing, indecomposable set in Lemma~\ref{lem-ac-cond-mc}.
Let $P_\mu$ be the transition probability function of the Markov chain $\{x_n\}$ on $(\X, \U(\X))$ induced by $\mu$.
We apply Lemma~\ref{lem-sub-sigalg} with $E_0 = \tilde \X$ and with $E_1, E_2, \ldots$ being a countable base of the topology on $\X$ (recall that $\X$ is separable and metrizable). This gives us a countably generated $\sigma$-algebra $\cE_\mu(\X) \subset \U(\X)$ such that $\B(\X) \subset \cE_\mu(\X)$, $\tilde \X \in \cE_\mu(\X)$, and $P_\mu$ restricted to $\cE_\mu(\X)$ is also a transition probability function.

Now let $\bar P_\mu$ and $\tilde P_\mu$ denote the restrictions of $P_\mu$ to $\U(\tilde \X) \times \tilde \X$ and to $\cE_\mu(\tilde \X) \times \tilde \X$, respectively, where the $\sigma$-algebra $\cE_\mu(\tilde \X) : = \{ E \mid E \in \cE_\mu(\X), E \subset \tilde \X\}$ and the $\sigma$-algebra $\U(\tilde \X)$ is likewise defined. Let $\{\bar x_n\}$ and $\{\tilde x_n\}$ be Markov chains on the state spaces $(\tilde \X, \U(\tilde \X) )$ and $(\tilde \X, \cE_\mu(\tilde \X))$ with transition probability functions $\bar P_\mu$ and $\tilde P_\mu$, respectively. We write $\Pr_x$ for the probability distribution of $\{\bar x_n\}$ or $\{\tilde x_n\}$ with initial state being $x$.

For clarity, we will now use different symbols to distinguish a measure from its completion: if $\phi$ is a measure on $\cE_\mu(\X)$, we write $\bar \phi$ for its completion or the restriction of its completion to $\U(\X)$; conversely, for a measure $\bar \phi$ on $\U(\X)$, we write $\phi$ for its restriction to the sub-$\sigma$-algebra $\cE_\mu(\X)$. In accordance with this notation, we refer to the measure $\psi$ in Lemma~\ref{lem-ac-cond-mc} as $\bar \psi$ instead.

To be concise, for two measures $\phi_1, \phi_2$, we use the shorthand notation $\phi_1 \ll \phi_2$ to mean that $\phi_1$ is absolutely continuous w.r.t.\ $\phi_2$.

\begin{lem} \label{lem-um-irr-mc}
The Markov chain $\{\bar x_n\}$ on $(\tilde \X, \U(\tilde \X))$ is $\bar \phi$-irreducible if and only if the Markov chain $\{\tilde x_n\}$ on $(\tilde \X, \cE_\mu(\tilde \X))$ is $\phi$-irreducible; and $\bar \psi$ is a maximal irreducibility measure of $\{\bar x_n\}$ if and only if $\psi$ is a maximal irreducibility measure of $\{\tilde x_n\}$.
\end{lem}

\begin{proof}
Clearly $\{\tilde x_n\}$ is $\phi$-irreducible if $\{\bar x_n\}$ is $\bar \phi$-irreducible. 
Conversely, suppose that $\{\tilde x_n\}$ is $\phi$-irreducible. 
For any set $B \in \U(\tilde \X)$ with $\bar \phi(B) > 0$, by Lemma~\ref{lem-umset}, there is a Borel set $\hat B \subset B$ with $\phi(\hat B) = \bar \phi(B) > 0$. Then by the irreducibility of $\{\tilde x_n\}$, $\Pr_x \{ \tilde x_n \in \hat B \ \text{for some} \ n \geq 1 \} > 0$ for all $x \in \tilde \X$. Since $B \supset \hat B$, this implies $\Pr_x \{ \bar x_n \in B \ \text{for some} \ n \geq 1 \} > 0$ for all $x \in \tilde \X$ and proves that $\{\bar x_n\}$ is $\bar \phi$-irreducible.

If $\bar \psi$ is a maximal irreducibility measure of $\{\bar x_n\}$ and $\phi$ is an irreducibility measure of $\{\tilde x_n\}$, then by the first part of the proof, $\bar \phi \ll \bar \psi$, which implies $\phi \ll \psi$, so $\psi$ is a maximal irreducibility measure of $\{\tilde x_n\}$. Conversely, if $\psi$ is a maximal irreducibility measure of $\{\tilde x_n\}$ and $\bar \phi$ is an irreducibility measure of $\{\bar x_n\}$, then by the first part of the proof, $\bar \psi$ is an irreducibility measure of $\{\bar x_n\}$ and $\phi$ is an irreducibility measure of $\{\tilde x_n\}$. The latter implies $\phi \ll \psi$, since $\psi$ is maximal for $\{\tilde x_n\}$. Then by Lemma~\ref{lem-umset}, we have $\bar \phi \ll \bar \psi$, so $\bar \psi$ is a maximal irreducibility measure of $\{ \bar x_n\}$.
\end{proof}

Henceforth, for irreducible Markov chains, the symbol $\psi$ or $\bar \psi$ will always stand for a maximal irreducibility measure.

The inequality (\ref{eq-sml-fn}) in the following lemma is called the \emph{minorization condition}. When it is satisfied, the function $s(\cdot)$ involved is called a \emph{small function}, and if $s(\cdot)$ is the indicator function for a set $C$, $C$ is called a \emph{small set} (cf.\ \cite[Def.~2.3]{Num84}). 
A large part of the theory of irreducible Markov chains requires the existence of a small function, which is ensured in the case of irreducible Markov chains on state spaces with countably generated $\sigma$-algebras. 
The theory becomes valid for irreducible Markov chains on $(\X, \U(\X))$ as well, if small functions exist in these Markov chains as well, which is shown to be true by the lemma below.

\begin{lem}[existence of a small function] \label{lem-small-fn}
Suppose that $\{\bar x_n\}$ is $\bar \psi$-irreducible. Then there exist a universally measurable function $s : \tilde \X \to [0, + \infty)$ with $\int \! s \, d \bar \psi > 0$, a nontrivial  $\sigma$-finite measure $\bar \nu$ on $\U(\tilde \X)$, a constant $\beta > 0$, and an integer $m_0 \geq 1$ such that
\begin{equation} \label{eq-sml-fn}
{\bar P}^{m_0}_\mu(B \mid x) \geq \beta s(x) \, \bar \nu(B), \qquad \forall \, x \in \tilde \X, \ B \in \U(\tilde \X).
\end{equation}
\end{lem}

\begin{proof}
By Lemma~\ref{lem-um-irr-mc}, the Markov chain $\{\tilde x_n\}$ on $(\tilde \X, \cE_\mu(\tilde \X))$ is $\psi$-irreducible.
Since $\cE_\mu(\tilde \X)$ is countably generated, by \cite[Thm.~2.1]{Num84}, 
there exist some integer $m_0 \geq 1$, constant $\beta > 0$, real-valued nonnegative $\cE_\mu(\tilde \X)$-measurable function $s(\cdot)$ with $\int \! s \, d \psi > 0$, and nontrivial $\sigma$-finite measure $\nu$ on $\cE_\mu(\tilde \X)$ such that
\begin{equation} \label{eq-prf-sml-fn}
  {\tilde P}^{m_0}_\mu(E \mid x) \geq \beta s(x) \, \nu(E), \qquad \forall \, x \in \tilde \X, \ E \in \cE_\mu(\tilde \X).
\end{equation}
Then $\int \! s \, d \bar \psi = \int \! s \, d \psi > 0$. For any $B \in \U(\tilde \X)$, by Lemma~\ref{lem-umset}, there exists Borel set $\hat B$ with $\hat B \subset B$ and $\nu(\hat B) = \bar \nu(B)$; therefore,  
for all $x \in \tilde \X$,
$$ {\bar P}_\mu^{m_0}( B \mid x) \geq {\bar P}_\mu^{m_0}(\hat B \mid x) =  {\tilde P}^{m_0}_\mu(\hat B \mid x) \geq \beta s(x) \, \nu(\hat B) = \beta s(x) \, \bar \nu(B),$$
where we used (\ref{eq-prf-sml-fn}) in the second inequality.
This proves (\ref{eq-sml-fn}).
\end{proof}

With the preceding lemma, we can now apply the theorems in the book \cite{Num84} for irreducible Markov chains to the Markov chain $\{\bar x_n\}$ on $(\tilde \X, \U(\tilde \X))$, alleviating the necessity for having countably generated $\sigma$-algebras on $\tilde \X$. 

\begin{proof}[Proof of Lemma~\ref{lem-ac-cond-mc}]
By assumption the Markov chain $\{\bar x_n\}$ is $\bar \psi$-irreducible and recurrent. Then by Lemma~\ref{lem-small-fn} and the preceding discussion, the results of \cite[Thm.~3.7 and Prop.~3.13]{Num84} are applicable to $\{\bar x_n\}$; they show that $\{\bar x_n\}$ has a unique maximal Harris set $\bar H$ and $\bar \psi(\tilde \X \setminus \bar H) = 0$.
Since $\{\bar x_n\}$ is Harris recurrent on $\bar H$ (cf.\ Appendix~\ref{appsec-mc} for the definition of a Harris set), 
it follows from the definition of Harris recurrence that condition (i) of Theorem~\ref{thm-ac-const} is satisfied if we let $\lambda = \bar \psi$ and $\hat \X = \bar H$ and take the stationary policy $\mu$ to be the required policy $\pi_B$ for every Borel set $B \subset \bar H$ with $\bar \psi(B) > 0$ and every $x \in \bar H$.
\end{proof}

\begin{proof}[Proof of Prop.~\ref{prp-ac-const-mc}]
By Lemma~\ref{lem-ac-cond-mc} and the assumptions of this proposition, condition (i) of Theorem~\ref{thm-ac-const} is satisfied for $\lambda = \bar \psi$ and $\hat \X = \bar H$. For Theorems~\ref{thm-ac-const}-\ref{thm-ac-const2} to hold with the replacements of some of their conditions as stated in this proposition, the only condition that remains to be verified is that $g^*, \tg^\star \not\equiv - \infty$ on $\bar H$. 
We verify this for $g^*$; the proof for $\tg^\star$ is similar.

Since $g^* \not\equiv - \infty$ on $\tilde \X$ by assumption, we have $E \cap \tilde \X \not= \varnothing$ for $E : = \{ x \in \X \mid g^*(x) > - \infty \}$.
By (\ref{eq-ac-ineq}), $g^*(x) \leq \int_\X g^*(y) P_\mu(dy \,|\, x)$ for all $x \in \X$, so $E$ is closed (i.e., $P_\mu(E^c \,|\, x) = 0$ for all $x \in E$) and hence its nonempty intersection with the closed set $\tilde \X$, $E  \cap  \tilde \X$, is also closed.
Since $\tilde \X$ is indecomposable, the two closed subsets of $\tilde \X$, $E \cap \tilde \X$ and $\bar H$, cannot be disjoint. Therefore 
$g^* \not\equiv - \infty$ on $\bar H$.
\end{proof}

\begin{proof}[Details for Remark~\ref{rmk-mc-related}(b)]
We explain the details for the comments about the case $g^*$ in this remark; the case $\tg^\star$ is similar.
Under the assumption that $g^*$ is bounded above and $g^* \not\equiv - \infty$ on $\tilde \X$, there is some constant $0 \leq \delta < \infty$ such that the function $f(x) : = \delta - g^*(x)$, $x \in \tilde \X$, is nonnegative and not everywhere infinite. Since $g^*$ is lower semianalytic by Theorem~\ref{thm-ac-basic}(i) and the set $\tilde \X \in \cE_\mu(\X)$ is universally measurable, $f$ is a measurable function on $(\tilde \X, \U(\tilde \X))$. Then, by Assumption~\ref{cond-ac-n1} and Lemma~\ref{lem-ineq}, $f$ is superharmonic for $\bar P_\mu$ (i.e., in operator notation, $f \geq \bar P_\mu f$).
Denote by $f_{\bar H}$ the restriction of $f$ to the maximal Harris set $\bar H$. The preceding proof of Prop.~\ref{prp-ac-const-mc} shows that $f_{\bar H} \not\equiv + \infty$ on $\bar H$, and therefore, $f_{\bar H}$ is a superharmonic function for the Harris recurrent Markov chain $\{\bar x_n\}$ on $\bar H$.
Then by a theorem on superharmonic functions for Harris recurrent Markov chains \cite[Thm.~3.8(i)]{Num84}, there is some constant $0 \leq t < \infty$ such that $f_{\bar H} = t$ $\bar \psi$-a.e.\ and $f_{\bar H} \geq t$ everywhere. 
Since $f = \delta - g^*$, for the finite constant $\ell : = \delta - t$, we have $g^* = \ell$ $\bar \psi$-a.e.\ and $g^* \leq \ell$ on $\bar H$.

For the second part of Remark~\ref{rmk-mc-related}(b), which involves a bounded $g^*$, the assumptions on the stationary policy $\mu$ and $g^*$ imply that for some constant $\delta > 0$, $f: = g^* + \delta$ is a bounded nonnegative function and 
harmonic for $P_\mu$ (i.e., in operator notation, $f = P_\mu f$). Since by assumption the Markov chain $\{x_n\}$ induced by $\mu$ is Harris recurrent, $f$ must be a constant \cite[Thm.~3.8(i)]{Num84} and hence $g^*$ is also a constant.
\end{proof}

%% appendix
\begin{appendices}

\section{Terminology for Markov Chains} \label{appsec-mc}
We briefly explain the concepts of $\psi$-irreducible, recurrent, and Harris recurrent Markov chains in this appendix. We follow the book \cite[Chaps.~2 and~3]{Num84}; for some general references on Markov chains, we refer the reader to this book and the book \cite{MeT09}. 
The definitions and results we mention below do not require the state space to have a countably generated $\sigma$-algebra. 

Let $(X, \Sigma)$ be a measure space. Consider a Markov chain $\{x_n\}$ on the state space $(X, \Sigma)$ with a probability transition function $P$. 
We call a nonempty set $E \in \Sigma$ \emph{closed} or \emph{absorbing} if $P(E^c \,|\, x) = 0$ for all $x \in E$.
The Markov chain $\{x_n\}$ restricted to such a set $E$ is a Markov chain with $(E, \Sigma(E))$ as its state space, where $\Sigma(E) = \{ B \mid B \in \Sigma, B \subset E\}$. A set $E \in \Sigma$ is called \emph{indecomposable} if $E$ does not contain two disjoint closed subsets.

Let $\phi$ be a nontrivial $\sigma$-finite measure on $(X, \Sigma)$. 
A Markov chain $\{x_n\}$ is called \emph{$\phi$-irreducible} and the measure $\phi$ is called an \emph{irreducibility measure} of $\{x_n\}$, if for every $E \in \Sigma$ with $\phi(E) > 0$,
\begin{equation} 
  \Pr_x \{ x_n \in E \ \text{for some} \ n \geq 1 \} > 0 \qquad \forall \, x \in X. \notag
\end{equation} 
Following \cite{Num84}, we call a Markov chain an \emph{irreducible Markov chain}, if it is $\phi$-irreducible for some $\phi$. 
Every irreducible Markov chain has a \emph{maximal irreducibility measure}, commonly denoted by $\psi$, in the sense that every irreducibility measure of the Markov chain is absolutely continuous w.r.t.\ $\psi$ (see \cite[Prop.~2.4]{Num84} or \cite[Prop.~4.2.2]{MeT09}). 
An irreducible Markov chain $\{x_n\}$ is called \emph{recurrent} if for all $E \in \Sigma$ with $\psi(E) > 0$,
$$ \Pr_x \{ x_n \in E \ \text{i.o.} \} > 0 \ \ \ \forall \, x \in X  \quad \text{and} \quad \Pr_x \{ x_n \in E \ \text{i.o.} \} = 1 \ \ \ \text{for $\psi$-almost all $x$} $$
(where ``i.o.''~stands for ``infinitely often'');
and \emph{Harris recurrent} if for all $E \in \Sigma$ with $\psi(E) > 0$,
$$ \Pr_x \{ x_n \in E \ \text{i.o.} \} = 1  \quad \forall \, x \in X.$$
An absorbing set $H \in \Sigma$ is called a \emph{Harris set} for $\{x_n\}$, if, restricted to $H$, the Markov chain $\{x_n\}$ is Harris recurrent.

A \emph{superharmonic} (resp.\ \emph{harmonic}) function for $P$ is a nonnegative measurable function $f$ on $(X, \Sigma)$ that is not identically infinite and satisfies $f \geq P f$ (resp.\ $f = P f$), in operator notation. If $\{x_n\}$ is Harris recurrent and $f$ is superharmonic, then for some finite constant $\ell \geq 0$, $f = \ell$ $\psi$-a.e.\ and $f \geq \ell$ everywhere; in particular, every bounded harmonic function is a constant \cite[Thm.~3.8]{Num84}.%\vspace*{-0.15cm} 
\end{appendices}

\section*{Acknowledgements}
The author would like to thank Professor Eugene Feinberg, who pointed her to several prior results on average-cost MDPs and provided helpful comments on an early draft of this paper, and Dr.~Martha Steenstrup, who read parts of this paper and gave her advice on improving the presentation.

%\vspace*{-0.25cm} 
\addcontentsline{toc}{section}{References} 
\bibliographystyle{apa} 
\let\oldbibliography\thebibliography
\renewcommand{\thebibliography}[1]{%
  \oldbibliography{#1}%
  \setlength{\itemsep}{0pt}%
}
{\fontsize{9}{11} \selectfont
\bibliography{ac_BorelDP_bib_new2}}

\end{document}